\documentclass[11pt]{amsart}
\usepackage{amsmath, amsfonts, amsthm, amssymb}
\usepackage[all,cmtip,line]{xy}
\usepackage{array}
\usepackage{enumerate}
\usepackage{srcltx} 

\newtheorem{thm}{Theorem}[section]
\newtheorem{lem}{Lemma}[section]
\newtheorem{prop}[lem]{Proposition}
\newtheorem{cor}[lem]{Corollary}
\newtheorem{q}[lem]{Question}
\newtheorem{defn}[lem]{Definition}
\newtheorem{rem}[lem]{Remark}

\numberwithin{equation}{section}
\newtheorem{assum}[lem]{Assumption}

\newcommand{\bR}{ \mathbb{R}} 
\newcommand{\bC}{ \mathbb{C}} 

\newcommand \tr{\;\text{tr}}
\newcommand \Hess{\;\text{Hess}}

\newcommand \eps{\varepsilon}

\newcommand \call{\mathcal L}

\newlength{\originalbase}
\title{viscosity solution to complex Hessian equations on compact Hermitian manifolds}

\author{Jingrui Cheng, Yulun Xu}

\date{\today}

\begin{document}
\maketitle
\begin{abstract}
We prove the existence of viscosity solutions to complex Hessian equations on a compact Hermitian manifold that satisfy a determinant domination condition.  This viscosity solution is shown to be unique when the right hand is strictly monotone increasing in terms of the solution.  When the right hand side does not depend on the solution,  we reduces it to the strict monotonicity of the solvability constant.
\end{abstract}
\section{introduction}
The goal of this note is to study the existence and uniqueness of viscosity solutions to complex Hessian equations on a closed Hermitian manifold.  There has been numerous works on the existence of viscosity solutions to complex Hessian equations,  as well as pluripotential solutions,  either on domains or on manifolds.  We refer the readers to \cite{DDT},  \cite{EGZ},  \cite{EGZ2}, \cite{HL2}, \cite{KNFSWW}, \cite{YW} and references therein for the quickly expanding literatures on this topic.  The techniques developed by the pioneering work of Guo-Phong-Tong \cite{GPT} allows us to develop stability estimates that make it possible to prove existence of weak solutions for more general complex Hessian equations.  On the other hand,  the regularization technique developed by our previous work \cite{CX} allows us to get a quite general uniqueness result.

Let $(M,\omega_0)$ be a closed Hermitian manifold.  In local coordinates,  we can write $\omega_0=\sqrt{-1}g_{i\bar{j}}dz_i\wedge d\bar{z}_j$.  
Let $\chi$ be a real $(1,1)$ form on $M$ and in local coordinates we can write it as: $\chi=\sqrt{-1}\chi_{i\bar{j}}dz_i\wedge d\bar{z}_j$.  For any $C^2$ function $\varphi:M\rightarrow \bR$,  we obtain a new real $(1,1)$ form: $\chi+dd^c\varphi$. We can define an operator $A: TX \rightarrow TX$ by $A^i_j=g^{i\bar{k}}\big(\chi_{j\bar{k}}+\varphi_{j\bar{k}}\big)$ in local coordinates. Let $\lambda[\chi+dd^c\varphi]$ be the (unordered) eigenvalues of $A$.  Equivalently,  $\lambda[\chi+dd^c\varphi]$ is the set of roots for:
\begin{equation*}
\det\big(\lambda g_{i\bar{j}}-(\chi_{i\bar{j}}+\varphi_{z_i\bar{z}_j})\big)=0.
\end{equation*}

Then we consider equations for $\varphi$ that may be written in the form:
\begin{equation}\label{1.1N}
F(\chi+dd^c\varphi)=h,\,\,\,h=e^{G(x)}\text{ or $e^{G(x,\varphi)}$}.
\end{equation}
In the above,  
\begin{equation}\label{1.2N}
F(\chi+dd^c\varphi)=f(\lambda[\chi+dd^c\varphi]),
\end{equation}
where $f(\lambda_1,\cdots,\lambda_n)$ is a smooth symmetric function.
$G(x)$ or $G(x,\varphi)$ are some right hand side one prescribes. The reason we write the right hand side in the form $e^G$ is to emphasize that it is strictly positive. Such equations have been
studied extensively in the literature, going back to the work of Caffarelli-Nirenberg-Spruck \cite{CNS} on the Dirichlet problem in the real case, when $\omega_0$ is the Euclidean metric
and $M$ is a domain in $\bR^n$.

We assume that the function $f$ in (\ref{1.2N}) is defined in a closed convex symmetric cone $\Gamma\subset \bR^n$,  with $\Gamma\subset \{\lambda\in \bR^n:\sum_{i=1}^n\lambda_i>0\}$ and containing the first octant $\Gamma_n$.  Therefore,  we need to assume that 
\begin{equation}\label{1.3N}
\lambda[\chi+dd^c\varphi]\in \Gamma.
\end{equation}

In addition,  we are going to assume that:
\begin{assum}\label{a1.1N}
\begin{enumerate}
\item $\frac{\partial f}{\partial \lambda_i}> 0,\,\,1\le i\le n$,  $f$ is concave and $>0$ on $\Gamma$ and $f=0$ on $\partial \Gamma$.
\item $\lambda[\chi](x)\in Int(\Gamma)$ for any $x\in M$.
\item (determinant domination condition) $f$ is a positive homogeneous function with degree 1,  and there exist constants $c_0>0$,  such that $f(\lambda)\ge c_0(\Pi_{i=1}^n\lambda_i)^{\frac{1}{n}}$ for $\lambda\in \Gamma_n=\{\lambda\in \bR^n:\lambda_i>0,\,\,1\le i\le n\}$.
\end{enumerate}
\end{assum}
In the above,  the assumption that $\frac{\partial f}{\partial\lambda_i}\ge 0$ implies the ellipticity of the equation (\ref{1.1N}),  (\ref{1.3N}).  The determinant domination condition is motivated by the pioneering work of Guo-Phong-Tong \cite{GPT},  which developed a unified PDE approach to $L^{\infty}$ estimate which satisfies:
\begin{equation}\label{1.4N}
\sum_i\lambda_i\frac{\partial f}{\partial \lambda_i}\le C_0f,\,\,\Pi_{i=1}^n\frac{\partial f}{\partial \lambda_i}\ge c_0,\text{ for some $c_0,\,C_0>0$}.
\end{equation}
\cite{GPT} also observed that determinant domination condition implies (\ref{1.4N}).   The assumption that $f$ is positive homogeneity one implies $\sum_i\lambda_i\frac{\partial f}{\partial \lambda_i}=f$.  The proof of the other property is contained in Lemma \ref{gpt assumption},  originally due to Guo-Phong-Tong \cite{GPT},  which we reproduce for the convenience of the readers.

There are many examples which satisfy the Assumption \ref{a1.1N} above.  The most well-known example is probably $f(\lambda)=\sigma_k^{\frac{1}{k}}(\lambda),\,1\le k\le n$,  defined on $\Gamma_k:=\{\lambda:\sigma_i(\lambda)\ge 0,\,1\le i\le k\}$,  where $\sigma_k(\lambda)$ is the $k$-th symmetric polynomial of $\lambda$.  These $\sigma_k$ equations have been extensively studied.  See \cite{DK}, \cite{DK2}, \cite{KN},  \cite{KN2} and references therein.  
However,  there are other examples satisfying Assumption \ref{a1.1N} which are less studied,  and we just name a few here:
\begin{enumerate}
\item ($\sigma_k$-equation for $(n-1)-$plurisubharmonic function) $f(\lambda)=\sigma_k^{\frac{1}{k}}(\tilde{\lambda})$,  where $\tilde{\lambda}_i=\frac{1}{n-1}\sum_{j\neq i}\lambda_j$.
\item ($p$-fold sum operator) $f(\lambda)=\big(\Pi_{|J|=p}\lambda_J\big)^{\frac{1}{N}}$,  where $\lambda_J=\lambda_{j_1}+\lambda_{j_2}+\cdots+\lambda_{j_p}$ and $N={n\choose p}$.
\end{enumerate}The first example with $k=n$ was studied by Tosatti-Weinkove \cite{TW} and the second example was considered by Harvey-Lawson \cite{HL2} related to $p$-geometry/$p$-potential theory.  Moreover,  we have the following general result due to Leonid Gurvits \cite{Gur} (see also \cite{HL} for a proof) that produces a large number of examples of $f$ satisfying the determinant domination condition.
\begin{prop}
Let $p(x)$ be a homogeneous polynomial of degree $N$ on $\bR^n$.  Denote $e=(1,\cdots,1)\in \bR^n$.  Assume that:
\begin{enumerate}
\item All coefficients of $p$ are $\ge 0$,
\item $p(e)>0$,
\item $\frac{\partial p}{\partial x_1}(e)=\frac{\partial p}{\partial x_2}(e)\cdots=\frac{\partial p}{\partial x_n}(e)=k$ for some $k>0$.
\end{enumerate}
Then $p(x)^{\frac{1}{N}}\ge c(x_1\cdots x_n)^{\frac{1}{n}}$ for some $c>0$ on $\{x_1>0,\cdots,x_n>0\}$.
\end{prop}
However,  the Inverse $\sigma_k$-equation does not satisfy the above determinant domination condition.

For a general Hessian equation without the determinant domination condition,  the apriori estimates will usually require the existence of a subsolution.  Székelyhidi \cite{G2} derived apriori estimates up to $C^{2,\alpha}$ assuming the existence of a $\mathcal{C}$-subsolution.  In order to prove existence of solutions,  the recent work of Guo-Song \cite{GS} shows that one needs a more delicate notion of subsolution.  It is no trivial issue to determine whether such subsolutions exist.
Assuming determinant domination condition alleviates this issue since we will always have 0 as a subsolution.  

Since there has been many works on the solvability in the smooth category,  it is a natural question to find weak solutions.  Most of the previous works have centered around $\sigma_k$-equation using pluripotential theory.  The work by Lu \cite{L} studied the existence and uniqueness of viscosity solutions to $\sigma_k$-equations on bounded domains in $\bC^n$ as well as homogeneous Hermitian manifolds.    We generalize this result and prove:
\begin{thm}\label{t1.1}
Assume that Assumption \ref{a1.1N} holds:
\begin{enumerate}
\item Let $G\in C(M)$,  then there exists a constant $c\in \bR$,  and $\varphi\in C(M)$ that solves the following equation in the viscosity sense:
\begin{equation*}
F(\chi+dd^c\varphi)=e^{G+c},\,\,\,\lambda[\chi+dd^c\varphi]\in \Gamma.
\end{equation*}
\item Let $G(x,u)\in C(M\times \bR)$.  Assume that $G$ is monotone increasing in $u$,  and $f(\lambda[\chi])(x)<e^{G(x,C_0)}$ for some $C_0\in \bR$ and any $x\in M$.  Then there exists $\varphi\in C(M)$ that solves the following equation in the viscosity sense:
\begin{equation*}
F(\chi+dd^c\varphi)=e^{G(x,\varphi)},\,\,\,\lambda[\chi+dd^c\varphi]\in \Gamma.
\end{equation*} 
\end{enumerate}
\end{thm}
\begin{rem}
The present result only applies to the case with strictly positive right hand side,  and we hope to deal with the degenerate case in subsequent works.
\end{rem}
When the right hand side is increasing with respect to $\varphi$, the uniqueness of the solution to K\"ahler-Einstein equation whose right hand side is in $L^p$ is proved in \cite{LPT}. For the uniqueness of solution to general Hessian equations,  we prove:
\begin{thm}
Let $G\in C(M\times \bR)$.  Assume that $G(x,u_1)<G(x,u_2)$ for any $x\in M$ and $u_1<u_2$.  Then there exists at most one viscosity solution to:
\begin{equation*}
F(\chi+dd^c\varphi)=e^{G(x,\varphi)},\,\,\lambda[\chi+dd^c\varphi]\in \Gamma.
\end{equation*}
\end{thm}
When the right hand side does not depend on $\varphi$,  the uniqueness is more subtle. For the complex Monge-Amp\'ere equation, the uniqueness is proved in \cite{KN} if the right hand side of the equation is in $L^p$ and has a positive lower bound. For the general Hessian equation,  as the first step,  we prove:
\begin{prop}
Assume that Assumption \ref{a1.1N} holds.
Let $G\in C(M)$,  then there exists a unique $c\in \bR$ such that the following equation is solvable in the viscosity sense:
\begin{equation*}
F(\chi+dd^c\varphi)=e^{G+c},\,\,\lambda[\chi+dd^c\varphi]\in \Gamma.
\end{equation*}
\end{prop}
We are not able to prove the uniqueness of the solution $\varphi$ and the main obstacle seems to be a lack of understanding of the constant $c$ that makes the equation solvable in the viscosity sense.  Indeed,  for $G\in C(M)$,  we may denote $c(G)$ to be the above said (unique) constant.  It is not very hard to see that $G_1\ge G_2$ implies $c(G_1)\le c(G_2)$.  The question that is of interest to us is whether this monotonicity is strict.  More precisely:
\begin{q}\label{q1.3}
Assume that $G_1\ge G_2$,  $G_1\neq G_2$.  Do we have $c(G_1)<c(G_2)$?
\end{q}
We show that an affirmative answer to Question \ref{q1.3} will lead to the uniqueness of viscosity solutions.  More precisely:
\begin{prop}
Let $G\in C(M)$.  Assume that for any $G'\le G,\,G'\neq G$ one has $c(G')>c(G)$,  then there is at most one solution to the following equation in the viscosity sense:
\begin{equation*}
F(\chi+dd^c\varphi)=e^{G+c},\,\,\lambda[\chi+dd^c\varphi]\in \Gamma,\,\,\sup_M\varphi=0.
\end{equation*}
\end{prop}
As to Question \ref{q1.3},  we observe that the answer is yes if $G_1$ or $G_2$ is smooth:
\begin{prop}
If $G_1$ or $G_2$ is smooth,  then the answer to Question \ref{q1.3} is affirmative.
\end{prop}

Next we explain our strategy of proof of the above results.  For the existence proof,  we first approximate the right hand side $G$ with a sequence of smooth right hand side $G_j$.  The apriori estimates developed in Székelyhidi \cite{G2} allows us to solve:
\begin{equation*}
F[\chi+dd^c \varphi_j]=e^{G_j(x)+c_j}\text{ or $e^{G_j(x,\varphi)}$}.
\end{equation*}
In order to get a viscosity solution,  all we need is to prove that $\varphi_j$ converges uniformly (at least up to a subsequence).  Using that $\varphi_j$ are all subharmonic and that $\varphi_j$ are uniformly bounded,  we see that $\varphi_j$ is precompact in $L^1$.  In order to improve the convergence to uniform convergence,  we need the following stability estimate (roughly stated): there exists $a>0$,  such that for any $v\in C^2(M)$,  $\lambda[\chi+dd^cv]\in \Gamma$,
\begin{equation}\label{1.5}
\sup_M(v-\varphi)\le C||(v-\varphi)_+||_{L^1}^a.
\end{equation}
The proof of (\ref{1.5}) is really a variant of the $L^{\infty}$ estimate by Guo-Phong-Tong \cite{GPT}.

For the uniqueness proof,  we need to consider the super/inf convolution,  adapted to manifolds.  We use the super/inf convolution considered in Cheng-Xu\cite{CX},  and show that it gives a semi-convex/concave approximation of the viscosity solutions.

Finally we explain the organization of this paper.  

In Section 2,  we explain some basic notations , definitions and some preliminary results that we need later on.

In Section 3,  we prove the existence of solution with smooth right hand side.

In Section 4,  we prove the stability result,  that allows us to improve the $L^1$ convergence to $L^{\infty}$ convergence,  thereby proving the existence of a viscosity solution.  

In Section 5,  we address the uniqueness issues.

\section{Notations and preliminaries}
In the following,  we denote $d^c=\frac{\sqrt{-1}}{2}(\bar{\partial}-\partial)$,  so that one has $\sqrt{-1}\partial\bar{\partial}=dd^c$.  The advantage of working with $d$ and $d^c$ is that they are real operators.

The central idea of viscosity solution is to use a $C^2$ test function to touch the solution from above and below,  which we define more precisely in the following:
\begin{defn}\label{touch}
Let $\varphi$ be a function defined on $M$ and $x_0\in M$.  Let $\psi$ be another function defined on an open subset of $M$ containing $x_0$.  
\begin{enumerate}
\item We say that $\psi$ touches $\varphi$ from above at $x_0$,  if there exists an open neighborhood $U$ of $x_0$ such that $\psi(x_0)=\varphi(x_0)$ and $\psi\ge\varphi$ on $U$,
\item We say that $\psi$ touches $\varphi$ from below at $x_0$,  if there exists an open neighborhood $U$ of $x_0$ such that $\psi(x_0)=\varphi(x_0)$ and $\psi\le\varphi$ on $U$.
\end{enumerate}
\end{defn}

The notion of viscosity solution we work with is consistent with Definition 1.5 in Crandall,  Ishii and Lions \cite{CMPL}:
\begin{defn}
Let $\varphi\in C(M)$ and $\Gamma$ be a closed convex symmetric cone in $\bR^n$ that contains the first octant.  Let $\tilde{\chi}$ be a real $(1,1)$ form on $M$.  We say that $\lambda[\tilde{\chi}+dd^c\varphi]\in \Gamma$ in the viscosity sense,  if for any $x_0\in M$,  and any $C^2$ function $P$ defined in a neighborhood of $x_0$ that touches $\varphi$ from above at $x_0$,  one has:
\begin{equation*}
\lambda[\tilde{\chi}+dd^cP](x_0)\in \Gamma.
\end{equation*}
We call such a function to be $\Gamma$-subharmonic with respect to $\tilde{\chi}$.
\end{defn}
Let $f(\lambda_1,\cdots,\lambda_n)$ and $\Gamma$ be as described in Section 1,  we will put $F(\chi+dd^c\varphi)=
f(\lambda[\chi+dd^c\varphi])$.  We may interchangably use both notations in the following.
\begin{defn}
\begin{enumerate}
    \item Let $\varphi$ be an upper semicontinuous function, We say that $\varphi$ is a viscosity subsolution to:
    \begin{equation*}
F(\chi+dd^c\varphi)=e^{G(x,\varphi)},\,\,\,\,\lambda[\chi+dd^c\varphi]\in \Gamma,
\end{equation*}
if for any $x_0\in M$ and any $C^2$ function $P$ defined in a neighborhood of $x_0$ that touches $\varphi$ from above at $x_0$,  one has:
\begin{equation*}
F(\chi+dd^cP)(x_0)\ge e^{G(x_0,P(x_0))},\,\,\, \lambda[\chi+ dd^c P](x_0)\in \Gamma.
\end{equation*}
\item Let $\varphi$ be a lower semicontinuous function, we say that $\varphi$ is a viscosity supersolution to 
\begin{equation*}
F(\chi +dd^c\varphi)=e^{G(x,\varphi)},\,\,\,\,\lambda[\chi+dd^c\varphi]\in \Gamma,
\end{equation*}
if for any $x_0\in M$ and any $C^2$ function $P$ defined in a neighborhood of $x_0$ that touches $\varphi$ from above at $x_0$,  one has either
\begin{equation*}
\lambda[\chi+ dd^c P](x_0)\in \Gamma\text{ and }F(\chi+dd^cP)(x_0)\le e^{G(x_0,P(x_0))} .
\end{equation*}
or
\begin{equation*}
    \lambda[\chi+ dd^c P](x_0)\notin \Gamma.
\end{equation*}
\item We say that a continuous function $\varphi$ is a viscosity solution to 
\begin{equation*}
F(\chi +dd^c\varphi)=e^{G(x,\varphi)},\,\,\,\,\lambda[\chi+dd^c\varphi]\in \Gamma,
\end{equation*}
if $\varphi$ is both a viscosity subsolution and a viscosity supersolution.
\end{enumerate}
\end{defn}

We will also need the following theorem of Gauduchon \cite{G}:
\begin{thm}
Let $\omega$ be a Hermitian metric on $M$.  Then there exists a unique function $v\in C^{\infty}(M)$ such that $\inf_Mv=0$,  $dd^c\big(e^{(n-1)v}\omega^{n-1}\big)=0$.
\end{thm}
In general,  let $\chi$ be a real $(1,1)$ form on $M$ with $\lambda[\chi]\in \Gamma$.  We choose coordinates on an open subset of $M$.  Define:$
F^{i\bar{j}}=\big(\frac{\partial F}{\partial h_{a\bar{b}}}(\chi)\big)^{-1}_{ji}.$  That is,  $F^{i\bar{j}}\frac{\partial F}{\partial h_{q\bar{j}}}(\chi)=\delta_{iq}$.  We define:
\begin{equation}\label{2.1NN}
\Omega=\sqrt{-1}F^{i\bar{j}}dz_i\wedge d\bar{z}_j.
\end{equation}
Then one can verify that the above defined $\Omega$ is actually independent of the choice of coordinates.  Moreover,  from ellipticity,  one sees that $\Omega>0$,  hence $\Omega$ defines a Hermitian metric.  Another thing we note that the function $\det g_{i\bar{j}}\det(\frac{\partial F}{\partial h_{i\bar{j}}}(\chi))$ is also independent of the choice of coordinates.  Moreover,  for any $u\in C^2(M)$,  the following formula holds:
\begin{equation}\label{2.2NNN}
\frac{\partial F}{\partial h_{i\bar{j}}}(\chi)u_{i\bar{j}}\frac{\omega_0^n}{n!}=\det g_{i\bar{j}}\det(\frac{\partial F}{\partial h_{i\bar{j}}}(\chi))\frac{\Omega^{n-1}}{(n-1)!}\wedge dd^cu.
\end{equation}

We want to mention the following Lemma which is the Lemma 4 in \cite{GPT}:
\begin{lem}\label{gpt assumption}
    Assume that $f: \mathbb{R}^n \rightarrow \mathbb{R}_+$ is a concave and homogeneous function of degree one, which satisfies $\frac{\partial f(\lambda)}{\partial \lambda_j}>0$ for any $\lambda$ in an admissible cone $\Gamma \subset \mathbb{R}^n$. Assume that there is a $\gamma>0$ such that 
    \begin{equation}\label{2.3NNew}
        f(\mu) \ge n \gamma^{\frac{1}{n}} (\Pi_j \mu_j)^{\frac{1}{n}}, \text{ for all } \mu \in \Gamma_n \triangleq \{\lambda \in \mathbb{R}^n: \lambda_1>0,...,\lambda_n>0\}.
    \end{equation}
    Then there exists a constant $\gamma>0$ such that $f$ satisfies the structural condition:
    \begin{equation}\label{e 2.3}
        \Pi_{i=1}^n\frac{\partial f}{\partial \lambda_i}(\lambda)\ge \gamma, \text{ for all }\lambda \in \Gamma.
    \end{equation}
\end{lem}
\begin{proof}
    By the concavity of $f$ on $\Gamma$, for any $\lambda, \mu \in \Gamma$ we have
    \begin{equation}\label{e 2.4}
        f(\mu)\le f(\lambda) +\sum_{j=1}^n (-\lambda_j + \mu_j) \frac{\partial f(\lambda)}{\partial \lambda_j} =\sum_{j=1}^n \mu_j \frac{\partial f(\lambda)}{\partial \lambda_j},
    \end{equation}
    where we have used the homogeneity of degree one assumption on $f$, which implies that $\sum_{j}\lambda_j \frac{\partial f(\lambda)}{\partial \lambda_j}=f(\lambda)$. Taking the infimum of the right hand side of $(\ref{e 2.4})$ over all $\mu \in \Gamma_n$ with $\Pi_{j=1}^n \mu_j=1$,  we see that:
\begin{equation*}
\inf_{\Pi_j\mu_j=1,\,\mu\in \Gamma_n}\big(\sum_j\mu_j\frac{\partial f(\lambda)}{\partial \lambda_j}\big)\ge \inf_{\Pi_j\mu_j=1,\,\mu\in \Gamma_n}f(\mu)\ge n\gamma^{\frac{1}{n}}.
\end{equation*}
The last inequality follows from (\ref{2.3NNew}).  On the other hand,  from the equality case of the arithmetic-geometric inequality,  we see that:
\begin{equation*}
\inf_{\Pi_j\mu_j=1,\,\mu\in\Gamma_n}\big(\sum_j\mu_j\frac{\partial f(\lambda)}{\partial \lambda_j}\big)=n\big(\Pi_j\frac{\partial f(\lambda)}{\partial \lambda_j}\big)^{\frac{1}{n}}.
\end{equation*}
Therefore we get:
\begin{equation*}
\Pi_j\frac{\partial f(\lambda)}{\partial \lambda_j}\ge \gamma.
\end{equation*}
\end{proof}
\section{existence of solution with smooth right hand side}
In this section,  our goal is to establish:
\begin{thm}\label{t2.1}
\begin{enumerate}
\item Assume that $G(x,u)$ is smooth,  and $G_{u}(x,u)>0$.  Then there exists a unique solution to
\begin{equation*}
F\big(\chi+dd^c\varphi\big)=e^{G(x,\varphi)}.
\end{equation*}
\item Assume that $G(x)$ is smooth,  then there is a unique $c\in \bR$ and $\varphi\in C^{\infty}(M)$ that solves:
\begin{equation*}
F(\chi+dd^c\varphi)=e^{G(x)+c},\,\,\sup_M\varphi=0.
\end{equation*}
\end{enumerate}
\end{thm}
The above result probably exists somewhere in the literature,  but we were unable to locate the exact reference.  
The uniqueness part (of both $\varphi$ and $c$) is an easy consequence of maximum principle. 

For the existence part,  we are going to use a continuity path.  
The continuity path when $G_u>0$ will be:
\begin{equation}\label{2.1N}
F(\chi+dd^c\varphi)=e^{(1-t)(\varphi+B_0(x))+tG(x,\varphi)},\,\,t\in[0,1].
\end{equation} 
If we denote $G(t,x,u)=(1-t)(u+B_0(x))+tG(x,u)$,  then we see that $G_{u}(t,x,u)$ remains positive.  In the above,  $B_0(x)$ is chosen so as to make sure $\varphi=0$ solves the equation with $t=0$,  namely:
\begin{equation*}
F(\chi)=e^{B_0(x)}.
\end{equation*}
When the right hand side does not depend on $\varphi$,  we are going to use the following continuity path:
\begin{equation}\label{2.2N}
F(\chi+dd^c\varphi)=e^{(1-t)B_0(x)+tG(x)+c_t},\,\,t\in[0,1].
\end{equation}
Note that in the above, there is a unique constant $c_t$ for which the solution could exist,  for each $t\in[0,1]$. Clearly $c_0=0$ and $\varphi=0$ when $t=0$.
It only remains to establish the openness and closedness.
\subsection{Openness of the coutinuity path}
First we consider the openness of (\ref{2.1N}).  We are going to set up the nonlinear mapping as follows:
\begin{equation*}
\begin{split}
\mathcal{F}:&\bR\times C^{k,\alpha}(M)\rightarrow C^{k-2,\alpha}(M)\\
&(t,\varphi)\mapsto \log\big(F(\chi+dd^c\varphi)\big)-((1-t)(\varphi+B_0(x))+tG(x,\varphi)).
\end{split}
\end{equation*}
In the above,  $k$ is sufficiently large.  
By implicit function theorem,  all we need is to verify that $D_{\varphi}\mathcal{F}(t_0,\varphi_{t_0})$ is an invertible map from $C^{k,\alpha}(M)$ to $C^{k-2,\alpha}(M)$,  where $t_0\in [0,1]$ and $\varphi_{t_0}$ is the solution to (\ref{2.1N}) corresponding to $t_0$.  We can compute that:
\begin{equation*}
\begin{split}
D_{\varphi}\mathcal{F}(t_0,\varphi_{t_0}):&C^{k,\alpha}(M)\rightarrow C^{k-2,\alpha}(M),\\
&\psi\mapsto \frac{1}{F(\chi+dd^c\varphi_{t_0})}\frac{\partial F}{\partial h_{i\bar{j}}}(\chi+dd^c\varphi_{t_0})\partial_{i\bar{j}}\psi-((1-t)+tG_u)\psi.
\end{split}
\end{equation*}
Let us denote $L=D_{\varphi}\mathcal{F}(t_0,\varphi_0)$ and we have:
\begin{lem}
The operator $L$ is invertible from $C^{k,\alpha}(M)$ to $C^{k-2,\alpha}(M)$.
\end{lem}
\begin{proof}
Injectivity is quite easy to see.  Indeed,  one can look at the point where $\psi$ achieves positive maximum or negative minimum.  One can take a coordinate chart near that point,  so that locally $Lu=0$ could be written as:
\begin{equation*}
a_{ij}(x)\partial_{ij}\psi+c\psi=0,\,\,c< 0.
\end{equation*}
We see by (strong) maximum principle that $u$ must be a constant and clearly this constant must be zero.

To prove surjectivity,  we only need to show that $L$ is surjective from $H^2$ to $L^2$.  With the right hand side in $C^{k-2,\alpha}(M)$,  one can use the standard elliptic regularity theory and work locally to improve the pre-image to be in $C^{k,\alpha}(M)$.
In order to show the surjectivity of $L$ from $H^2$ to $L^2$,  it will not work well to consider $L$,  but we will need to consider $\tilde{L}=e^{(n-1)v_0}L$,  where $v_0$ is the Gauduchon factor of the following Hermitian metric:
\begin{equation}\label{2.3NN}
\tilde{\Omega}=F^{-\frac{1}{n-1}}(\chi+dd^c\varphi_{t_0})(\det g\det(\frac{\partial F}{\partial h_{i\bar{j}}})(\chi+dd^c\varphi_{t_0}))^{\frac{1}{n-1}}\Omega.
\end{equation}
That is,  $dd^c\big(e^{(n-1)v_0}\tilde{\Omega}^{n-1}\big)=0$.  In the above,  $\Omega$ is defined in (\ref{2.1NN}),  evaluated at $\chi+dd^c\varphi_{t_0}$.  The existence of $v_0$ was first proved by Gauduchon \cite{G} (A modern proof was given by Fu-Wang-Wu in \cite{FWW}).
We will be able to show that $\tilde{L}$ is surjective (so $L$ is surjective) if we can show that:
\begin{enumerate}
\item $Im(\tilde{L})$ is a closed subspace of $L^2$.  For this we just need to show $||\psi||_{H^2}\le C||\tilde{L}\psi||_{L^2}$.
\item $Ker(\tilde{L}^*)=0$,  where $\tilde{L}^*:L^2\rightarrow H^{-2}$ is the adjoint map of $\tilde{L}$ and $H^{-2}$ is the dual space of $H^2$.
\end{enumerate}

To prove (1),  first we observe that $L$ can be re-written as:
\begin{equation*}
L\psi=\frac{1}{F}\frac{\det g\det(\frac{\partial F}{\partial h_{i\bar{j}}})\frac{\Omega^{n-1}}{(n-1)!}\wedge dd^c\psi}{\frac{1}{n!}\omega_0^n}-((1-t)+tG_u)\psi.
\end{equation*}
Therefore
\begin{equation}\label{2.3}
(\tilde{L}\psi)\frac{1}{n!}\omega_0^n=e^{(n-1)v_0}\tilde{\Omega}^{n-1}\wedge dd^c\psi-e^{(n-1)v_0}((1-t)+tG_u)\psi\frac{1}{n!}\omega_0^n.
\end{equation}
Now we multiply both sides by $\psi$ and integrate,  we see that:
\begin{equation*}
\begin{split}
&\int_M\psi\tilde{L}\psi\frac{1}{n!}\omega_0^n=\int_Me^{(n-1)v_0}\tilde{\Omega}^{n-1}\wedge (\psi dd^c\psi)-\int_Me^{(n-1)v_0}((1-t)+tG_{u})\psi^2\frac{\omega_0^n}{n!}\\
&=\int_Me^{(n-1)v_0}\tilde{\Omega}^{n-1}\wedge(dd^c(\frac{\psi^2}{2})-d\psi\wedge d^c\psi)-\int_Me^{(n-1)v_0}((1-t)+tG_{u})\psi^2\frac{\omega_0^n}{n!}\\
&=-\int_Me^{(n-1)v_0}\tilde{\Omega}^{n-1}\wedge d\psi\wedge d^c\psi-\int_Me^{(n-1)v_0}((1-t)+tG_{u})\psi^2\frac{\omega_0^n}{n!}.
\end{split}
\end{equation*}
At this point,  Using the Cauchy-Schwarz inequality,  we see that:
\begin{equation}\label{2.4NN}
||\psi||_{L^2}+||\nabla \psi||_{L^2}\le C||\tilde{L}\psi||_{L^2}.
\end{equation}
To get the estimate for the second derivative,  we go back to (\ref{2.3}) and re-write it as:
\begin{equation*}
\begin{split}
&d\big(e^{(n-1)v_0}\tilde{\Omega}^{n-1}\wedge d^c\psi\big)=(\tilde{L}\psi)\frac{1}{n!}\omega_0^n+d\big(e^{(n-1)v_0}\tilde{\Omega}^{n-1}\big)\wedge d^c\psi\\
&+e^{(n-1)v_0}((1-t)+tG_{u})\psi\frac{\omega_0^n}{n!}.
\end{split}
\end{equation*}
Note that the right hand side is in $L^2$ now.  Writing this equation in local coordinates would be of the form:
\begin{equation*}
\partial_i(a_{ij}\partial_ju)=k,\,\,k\in L^2.
\end{equation*}
The standard estimate (see,  for example,  Evans \cite{Evans},  Section 6.3,  Theorem 1) would give: $||u||_{H^2(B_{\frac{1}{2}})}\le C(||u||_{L^2(B_1)}+||k||_{L^2(B_1)})$.  This estimate combined with (\ref{2.4NN}) gives what we need.

Next we show that $\tilde{L}^*$ is injective.  we can compute that:
\begin{equation*}
(\tilde{L}^*v)=\frac{dd^c\big(ve^{(n-1)v_0}\tilde{\Omega}^{n-1}\big)}{\frac{1}{n!}\omega_0^n}-e^{(n-1)v_0}((1-t)+tG_{u})v,\,\,v\in L^2.
\end{equation*}
If $\tilde{L}^*v=0$,  then this equation written locally takes the form:
\begin{equation*}
\partial_{ij}(a_{ij}v)-hv=0,\,\,h\in C^{\infty}(B_1),a_{ij}\in C^{\infty},\,\,v\in L^2.
\end{equation*}
Using elliptic theory,  we can improve the regularity of $v$ to $C^{\infty}$,  but then $v$ satisfies:
\begin{equation*}
e^{(n-1)v_0}\tilde{\Omega}^{n-1}\wedge dd^cv+2dv\wedge d^c(e^{(n-1)v_0}\tilde{\Omega}^{n-1})-e^{(n-1)v_0}((1-t)+tG_{u})v\frac{\omega_0^n}{n!}=0.
\end{equation*}
In local coordinates,  the above equation reads:
\begin{equation*}
a_{ij}\partial_{ij}v+b_i\partial_iv+cv=0,\,\,c<0.
\end{equation*}
If one looks at this equation in a neighborhood of the point where $v$ achieves positive maximum or negative minimum,  we see from strong maximum principle that $v=0$.
\end{proof}
Now we consider the openness of (\ref{2.2N}).  We set up the nonlinear mapping in a similar way:
\begin{equation*}
\begin{split}
\mathcal{F}&:\bR\times \bR\times C_0^{k,\alpha}(M)\rightarrow C^{k-2,\alpha}(M)\\
&(t,c,\varphi)\mapsto \log\big(F(\chi+dd^c\varphi)\big)-(1-t)B_0(x)-tG(x)-c.
\end{split}
\end{equation*}
In the above:
\begin{equation*}
C_0^{k,\alpha}(M)=\{h\in C^{k,\alpha}(M):\int_Mh\omega_0^n=0\}.
\end{equation*}
Assume that (\ref{2.2}) is solvable with $t=t_0$,  that is,  there exist $(c_{t_0},\varphi_{t_0})$ such that $\mathcal{F}(t_0,c_{t_0},\varphi_{t_0})=0$.  By implicit function theorem,  we just need to verify the linearized map $D_{(c,\varphi)}\mathcal{F}$ defines a bijective map from $\bR\times C_0^{k,\alpha}(M)$ to $C^{k-2,\alpha}(M)$.  One can compute that: 
\begin{equation*}
\begin{split}
D_{(c,\varphi)}\mathcal{F}&:\bR\times C_0^{k,\alpha}(M)\rightarrow C^{k-2,\alpha}(M)\\
&(\lambda,u)\mapsto \frac{1}{F(\chi+dd^c\varphi_{t_0})}\frac{\partial F}{\partial h_{i\bar{j}}}(\chi+\sqrt{-1}\partial\bar{\partial}\varphi_{t_0})\partial_{i\bar{j}}u-\lambda.
\end{split}
\end{equation*}
Denote this linear operator to be $\call$.  We then have:
\begin{lem}
$\call$ is bijective from $\bR\times C_0^{k,\alpha}(M)$ to $C^{k-2,\alpha}(M)$.
\end{lem}
\begin{proof}
First,  similar as before,  $\call$ may be written as:
\begin{equation*}
\mathcal{L}(\lambda,u)=\frac{1}{F}\frac{\det g\det(\frac{\partial F}{\partial h_{i\bar{j}}})\frac{\Omega^{n-1}}{(n-1)!}\wedge dd^cu}{\frac{1}{n!}\omega_0^n}-\lambda=\frac{\frac{\tilde{\Omega}^{n-1}}{(n-1)!}\wedge dd^cu}{\frac{1}{n!}\omega_0^n}-\lambda.
\end{equation*}
Here $\tilde{\Omega}$ is defined by (\ref{2.3NN}).  Assume that $\mathcal{L}(\lambda,u)=0$,  we need to show that $\lambda=0,\,u=0$.
Let $v_0$ be the Gauduchon factor of $\tilde{\Omega}$,  so that $dd^c(e^{(n-1)v_0}\tilde{\Omega}^{n-1})=0$.  Multiplying by $e^{(n-1)v_0}$ and integrating gives:
\begin{equation*}
0=\int_Me^{(n-1)v_0}\frac{\tilde{\Omega}^{n-1}}{(n-1)!}\wedge dd^cu-\int_M\lambda e^{(n-1)v_0}\frac{\omega_0^n}{n!}=-\lambda\int_Me^{(n-1)v_0}\frac{\omega_0^n}{n!}.
\end{equation*}
This proves $\lambda=0$,  so that one has $\tilde{\Omega}^{n-1}\wedge dd^cu=0$.  Then one can use the strong maximum principle (similar to the argument of part (1)) to show that $u$ is a constant.  But we are requiring its integral to be equal to zero,  so $u=0$.

Next we show that $\mathcal{L}$ is surjective.  Just as before,  it will be sufficient to show that $\mathcal{L}$ is surjective from $\bR\times H_0^2$ to $L^2$.  Here $H_0^2$ denotes the element $u\in H^2$ such that $\int u\omega_0^n=0$.  Also it will be necessary to consider $\tilde{\call}(\lambda,u)=e^{(n-1)v_0}\tilde{\call}$.  We will be able to show $\tilde{\call}$ is surjective (hence $\call$),  if we can show:
\begin{enumerate}
\item $Im(\tilde{\call})$ is a closed subspace of $L^2$.  For this we just need to show that $||u||_{H^2}+|\lambda|\le C||\tilde{\call}(u,\lambda)||_{L^2}$.
\item $Ker(\tilde{\call}^*)=0$,  where $\tilde{\call}^*:L^2\rightarrow H^{-2}\times \bR$ is the adjoint map of $\call$.
\end{enumerate}
To see (1),  first note that one has:
\begin{equation}\label{2.6NN}
\tilde{\call}(\lambda,u)\frac{\omega_0^n}{n!}=e^{(n-1)v_0}\frac{\tilde{\Omega}^{n-1}}{(n-1)!}\wedge dd^cu-\lambda e^{(n-1)v_0}\frac{\omega_0^n}{n!}.
\end{equation}
Integrating both sides,  we get $|\lambda|\le C||\tilde{\call}(\lambda,u)||_{L^1}\le C'||\tilde{\call}(\lambda,u)||_{L^2}$.
Then we multiply (\ref{2.6NN}) by $u$,   and use that $udd^cu=dd^c(\frac{u^2}{2})-du\wedge d^cu$,  we integrate and find that:
\begin{equation*}
\int_Mu\tilde{\mathcal{L}}(\lambda,u)\frac{\omega_0^n}{n!}=-\int_Me^{(n-1)v_0}\tilde{\Omega}^{n-1}(n-1)!\wedge du\wedge d^cu-\int_M\lambda u e^{(n-1)v_0}\frac{\omega_0^n}{n!}.
\end{equation*}
This way we get:
\begin{equation*}
||\nabla u||_{L^2}^2\le \eps||u||_{L^2}^2+C_{\eps}(|\lambda|^2+||\tilde{\call}(\lambda,u)||_{L^2}^2),\,\,\,\forall \eps>0.
\end{equation*}
On the other hand,  since $\int_Mu\omega_0^n=0$,  we may use Poincaré inequality to see that $||\nabla u||_{L^2}\ge c_0||u||_{L^2}$.  This way we obtain that
\begin{equation*}
||u||_{L^2}+||\nabla u||_{L^2}\le C||\tilde{\call}(\lambda,u)||_{L^2}.
\end{equation*}
Then we get the second derivative estimates by writing (\ref{2.6NN}) under local coordinates,  and argue in the same way as part (1).

Next,  we can find the adjoint map is:
\begin{equation*}
\begin{split}
\tilde{\call}^*&:L^2\rightarrow \bR\times H^{-2}\\
&g\mapsto (-\int_Mge^{(n-1)v_0}\frac{\omega_0^n}{n!},dd^c\big(ge^{(n-1)v_0}\frac{\tilde{\Omega}^{n-1}}{(n-1)!}\big)).
\end{split}
\end{equation*}
If $dd^c\big(ge^{(n-1)v_0}\tilde{\Omega}^{n-1}\big)=0$ with $g\in L^2$,  then one can improve the regularity of $g$ and see that $g\in C^{\infty}$.  But then 
\begin{equation*}
0=\int_Mgdd^c(ge^{(n-1)v_0}\tilde{\Omega}^{n-1})=-\int_Me^{(n-1)v_0}dg\wedge d^c g\wedge \tilde{\Omega}^{n-1}.
\end{equation*}
This implies $g$ is a constant.  On the other hand,  from $\int_Mge^{(n-1)v_0}\omega_0^n=0$,  we see that $g=0$.  This proves $\tilde{\call}^*$ is injective.
\end{proof}

\subsection{Closedness of the continuity path}
The required apriori estimates we need can be found in Székelyhidi \cite{G2},  who proved that:  
\begin{thm}\label{Szekelyhidi}
Consider the Hessian equation $f\big(\lambda[\chi+\sqrt{-1}\partial\bar{\partial}u]\big)=h(x)$,  where $f$ and $h$ satisfy the following assumptions:
\begin{enumerate}
\item $\frac{\partial f}{\partial \lambda_i}>0$ and $f$ is convex.
\item $\sup_{\partial \Gamma}f<\inf_Mh$,
\item For any $\sigma<\sup_{\Gamma}f$ and $\lambda\in \Gamma$ we have $\lim_{t\rightarrow \infty}f(t\lambda)>\sigma$.
\end{enumerate}
Suppose $u$ is a (smooth) solution with $\sup_Mu=0$ and $\underline{u}$ is a $\mathcal{C}$-subsolution,  then we have an estimate $||u||_{2,\alpha}\le C$,  where $C$ depends on the given data $M,\,g,\,\chi,\,h$ and the subsolution $\underline{u}$.
\end{thm}
We are going to apply this result to (\ref{2.2N}).  Note that by evaluating at the minimum and maximum of $\varphi$,  we see that $c_t$ is actually uniformly bounded in $t$.  So we would be able to apply Theorem \ref{Szekelyhidi} if we can show the existence of a $\mathcal{C}$-subsolution.

Indeed,  a $\mathcal{C}$-subsolution simply means a (smooth) function $\underline{u}$,  such that for each $x\in M$,  the following set is bounded:
\begin{equation}\label{2.2}
\{\lambda'\in \Gamma:f(\lambda')=h(x),\,\,\text{ and }\lambda'-\lambda(\chi+\sqrt{-1}\partial\bar{\partial}\underline{u})\in \Gamma_n\}.
\end{equation}
 In our setting,  the situation is simple and we are going to see that $\underline{u}=0$ will be a $\mathcal{C}$-subsolution.  More precisely:
\begin{lem}\label{l2.3}
For any $C_0>0$ and any $x\in M$,  we define the following set: 
\begin{equation*}\{\lambda'\in \Gamma:f(\lambda')\le C_0,\text{ and }\lambda'-\lambda(\chi)(x)\in \Gamma_n\}.
\end{equation*}
This set is bounded and one can estimate the diameter of this set in terms of $f$,  $C_0$,  $\chi$ and the background metric $g$.
\end{lem}
\begin{proof}
Denote $\Gamma_{\infty}=\{(\lambda_1,\cdots,\lambda_{n-1}):\text{$(\lambda_1,\cdots,\lambda_n)\in \Gamma$ for some $\lambda_n$}\}.$ For any $\lambda'=(\lambda_1,\cdots,\lambda_{n-1})\in \Gamma_{\infty}$,  we consider the limit
\begin{equation*}
\lim_{\lambda_n\rightarrow +\infty}f(\lambda_1,\cdots,\lambda_n).
\end{equation*}
Trudinger \cite{T} proved that either this limit is infinite for all $\lambda'\in \Gamma_{\infty}$ or this limit is finite for all $\lambda'\in \Gamma_{\infty}$.  Moreover,  this convergence is uniform on any compact subset of $\Gamma_{\infty}$.  The proof essentially follows from the concavity of $f$.  In our setting,  we will have the above limit is infinite for all $\lambda'\in \Gamma_{\infty}$.  Indeed,  since $f$ satisfies the determinant domination condition,  we see that,  for $\lambda_1>0,\cdots,\lambda_{n-1}>0$,  we have: $f(\lambda_1,\cdots,\lambda_n)\ge c\big(\Pi_{i=1}^n\lambda_i\big)^{\frac{1}{n}}\rightarrow +\infty$ as $\lambda_n\rightarrow +\infty$.

From this,  we see that,  for each $1\le i\le n$,  we have:
\begin{equation*}
\lim_{t\rightarrow +\infty}f\big(\lambda[\chi]+t\mathbf{e}_i\big)=+\infty,
\end{equation*}
where $\mathbf{e}_i$ is the standard basis in $\bR^n$.  Moreover,  since $\lambda[\chi]$ is strictly contained in $Int(\Gamma)$ as $x$ varies over $M$,  the above convergence is uniform on $M$.  Using the concavity of $f$ again,  we see that:
\begin{equation*}
\lim_{\mu\in \Gamma_n,\,\mu\rightarrow \infty}f\big(\lambda[\chi]+\mu\big)=+\infty,
\end{equation*}
and this convergence is uniform on $M$.
\end{proof}
Therefore,  we may use Theorem \ref{Szekelyhidi} to conclude that $||\varphi||_{2,\alpha}$ is uniformly bounded in (\ref{2.2N}).  From the standard elliptic theory,  we see that $\varphi$ is uniformly bounded in any higher order norm.

Now let us consider the continuity path (\ref{2.1N}).  First we show that $\varphi$ is bounded from above,  uniform for $t\in [0,1]$.  Indeed,  we fix any $t\in[0,1]$,  and assume that $\varphi$ achieves maximum at $x_0$,  then at this point,  we have
\begin{equation}\label{2.8}
F(\chi)(x_0)\ge F(\chi+dd^c\varphi)(x_0)=e^{(1-t)(\varphi+B_0(x_0))+tG(x_0,\varphi)}.
\end{equation}
Our assumption on $G$ was that there exists $C_0>0$ large enough,  such that for all $\varphi\ge C_0$,  one has
\begin{equation*}
F(\chi)(x)<e^{G(x,\varphi)},\,\,\varphi\ge C_0.
\end{equation*}
Therefore,  there exists $C_0'$,  such that for any $t\in [0,1],\,\varphi\ge C_0'$,  one has:
\begin{equation}\label{2.9}
F(\chi)(x)<e^{(1-t)(\varphi+B_0(x))+tG(x,\varphi)}.
\end{equation}
(\ref{2.8}) and (\ref{2.9}) combined shows that $\varphi(x_0)\le C_0'$,  hence $\sup_M\varphi\le C_0',\,\,t\in[0,1]$.

Once we have a bound of $\sup_M\varphi$,  we see that the right hand side of (\ref{2.1N}) is uniformly bounded.  Using Lemma \ref{l2.3},  we see that the following set is uniformly bounded in $x\in M$ and $t\in [0,1]$:\begin{equation*}
\{\lambda'\in \Gamma:\,\,f(\lambda')=e^{(1-t)(\varphi+B_0(x))+tG(x,\varphi)}\text{ and $\lambda'-\lambda[\chi]\in \Gamma_n$}\}
\end{equation*}
 Therefore,  Theorem \ref{Szekelyhidi} carries over in this case (indeed,  the monotone increasing dependence on $\varphi$ actually helps with the estimates) and gives $||\varphi||_{2,\alpha}$ is uniformly bounded.  From standard elliptic estimates,  we see that $\varphi$ is uniformly bounded in any higher order norm.

\section{stability estimate and existence of viscosity solutions}
The key stability result we need is the following:
\begin{prop}\label{p3.1}
Let $v$ be a bounded smooth $\Gamma$-subharmonic function,  and let $\varphi$ be a smooth solution to (\ref{1.1N}),  (\ref{1.3N}) with $h=e^{G_0(x)}$ and $\sup_M\varphi=0$.  Then for any $p_0>1$,  and any $a<\frac{p_0-1}{np_0+p_0-1}$ we have
\begin{equation*}
||(v-\varphi)_+||_{L^{\infty}}\le C||(v-\varphi)_+||_{L^1}^{a}.
\end{equation*}
Here $C$ depends only on $||v||_{L^{\infty}}$,  $||e^{nG_0}||_{L^{p_0}}$,  the background metric and the choice of $a$.
\end{prop}
\begin{rem}
The above estimate uses the $L^{p},\,p>n$ norm of the right hand side $e^{G_0}$.  This is more than enough for us at the moment,  but it may be needed in the future.
\end{rem}
Let us postpone the proof for the moment,  and we first use this proposition to prove existence.  

First we look at the case when $G$ depends only on $x$.
We can take a sequence of smooth $G_j(x)$ that converges to $G(x)$ uniformly.  By Theorem \ref{t2.1},  there exists $\varphi_j\in C^{\infty}(M)$ which is strictly $\Gamma$-subharmonic,  that solves:
\begin{equation}\label{3.1}
F(\chi+dd^c \varphi_j)=e^{c_j+G_j(x)},\,\,\sup_M\varphi_j=0.
\end{equation}
Our first goal will be to show that $\varphi_j$ is uniformly bounded.  For this we first need to estimate the constants $c_j$.  Evaluating at the maximum and minimum point of $\varphi_j(x)$ respectively,  we find that:
\begin{equation*}
|c_j|\le \max|G_j|+\max |\log F(\chi)|.
\end{equation*}
Since $G_j$ approximates $G$ uniformly,  we see that $\max|G_j|$,  hence $c_j$ is uniformly bounded.  Next,  the uniform $C^0$ bound of $\varphi_j$ follows from Székelyhidi's work \cite{G2}. 

Now we look at the situation when $G(x,u)$ is monotone increasing in $u$.  We can approximate $G(x,u)$ uniformly by a sequence of $G_j(x,u)$,  such that:
\begin{enumerate}
\item Each $G_j(x,u)$ is $C^{\infty}$ smooth in $x$ and $u$,
\item $\frac{\partial G_j}{\partial u}(x,u)>0$,
\item There exists $C_0>0$,  such that for all $j$,  $x\in M$ and $u>C_0$,  one has $F(\chi)(x)<e^{G_j(x,u)}$.
\end{enumerate}
 We can use Theorem \ref{t2.1} to conclude that one can solve:
\begin{equation}\label{3.2}
F(\chi+dd^c\varphi_j)=e^{G_j(x,\varphi_j)}.
\end{equation}
Evaluating at the maximum point of $\varphi_j$,  from point (3) above we see that 
\begin{equation*}
\sup_M\varphi_j\le C_0.
\end{equation*}
Once we have the uniform upper bound of $\varphi_j$,  we see that the right hand side $e^{G_j(x,\varphi_j)}$ is uniformly bounded from above,  from which we can deduce a uniform lower bound for $\varphi_j$ following Székelyhidi \cite{G2} or Guo-Phong-Tong \cite{GPT}.

 Next we show that there is uniform $L^2$ bound for $\nabla\varphi$:
\begin{lem}
There exists a constant $C$,  depending only on the $C^0$ bound of $G$,  and also the Hermitian metric $\omega_0$,  such that
\begin{equation*}
\int_M d\varphi_j\wedge d^c\varphi_j\wedge \omega_0^{n-1}\le C.
\end{equation*}
\end{lem}
\begin{proof}
We use the fact that each $\varphi_j$ is strictly $\Gamma$-subharmonic,  and that $\Gamma\subset \Gamma_1:=\{\lambda\in \bR^n:\sum_{i=1}^n\lambda_i\ge 0\}$.
This would imply that:
\begin{equation*}
\Delta_{\omega_0}\varphi_j\ge -n,\,\,\,\text{ where $\Delta_{\omega_0}\varphi_j=g^{p\bar{q}}\partial_{p\bar{q}}\varphi_j$}.
\end{equation*}
This is equivalent to:
\begin{equation*}
dd^c\varphi_j\wedge \omega_0^{n-1}\ge -\omega_0^n.
\end{equation*}
Let $C_1>0$ be large enough such that $\varphi_j+C_1\ge 0$.  Then you multiply this to both sides above:
\begin{equation*}
-\int_Md\varphi_j\wedge d^c\varphi_j\wedge \omega_0^{n-1}+\int_M(\varphi_j+C_1)d^c\varphi_j\wedge d\omega_0^{n-1}\ge -\int_M(\varphi_j+C_1)\omega_0^n.
\end{equation*}
For the middle term,  one has
\begin{equation*}
\begin{split}
&\int_M(\varphi_j+C_1)d^c\varphi_j\wedge d\omega_0^{n-1}=\int_Md^c\big(\frac{(\varphi_j+C_1)^2}{2}\big)\wedge d\omega_0^{n-1}=-\int_Md\big(\frac{(\varphi_j+C_1)^2}{2}\big)\wedge d^c\omega_0^{n-1}\\
&=\int_M\frac{(\varphi_j+C_1)^2}{2}dd^c\omega_0^{n-1}.
\end{split}
\end{equation*}
Therefore,  we see that the integral $\int_Md\varphi_j\wedge d^c\varphi_j\wedge \omega_0^{n-1}$ is uniformly bounded.
\end{proof}
Hence we are in a position to apply the Rellich compact embedding theorem to conclude that there is a subsequence $\varphi_{j_k}$ that converges in $L^1$.

Before we proceed with the proof,  first we explain how to use this to obtain existence:
\begin{cor}
Theorem \ref{t1.1} holds.
\end{cor}
\begin{proof}
First we prove part (1).  We have shown that one can find smooth solutions to 
\begin{equation*}
F(\chi+dd^c\varphi_j)=e^{G_j+c_j},\,\,\lambda[\chi+dd^c\varphi_j]\in \Gamma,\,\,\sup_M\varphi_j=0.
\end{equation*}

Now we can use Proposition \ref{p3.1},  and take $v=\varphi_{j_l}$,  $\varphi=\varphi_{j_k}$ to get:
\begin{equation*}
||(\varphi_{j_l}-\varphi_{j_k})_+||_{L^{\infty}}\le C||(\varphi_{j_l}-\varphi_{j_k})_+||_{L^1}^{a}.
\end{equation*}
Switching the choice of $v$ and $\varphi$ gives:
\begin{equation*}
||\varphi_{j_k}-\varphi_{j_l}||_{L^{\infty}}\le C||\varphi_{j_k}-\varphi_{j_l}||_{L^1}^{a}.
\end{equation*}
That is,  the subsequence $\varphi_{j_k}$ actually converges uniformly.  Therefore,  it is easy to see that their uniform limit will solve the limit equation in the viscosity sense.

Now we look at part (2).  We have shown that one can find smooth solutions to
\begin{equation*}
F(\chi+dd^c\varphi_j)=e^{G_j(x,\varphi_j)},\,\,\lambda[\chi+dd^c\varphi_j]\in \Gamma.
\end{equation*}
Moreover,  we have seen that $\varphi_j$ has uniform $C^0$ bound.  Denote $\tilde{c}_j=\sup_M\varphi_j$ and define $\tilde{\varphi}_j=\varphi_j-\tilde{c}_j$,  so that 
$\tilde{\varphi}_j$ solves:
\begin{equation*}
F(\chi+dd^c\tilde{\varphi}_j)=e^{G_j(x,\tilde{c}_j+\tilde{\varphi}_j)}.
\end{equation*}
We have seen that one can take a subsequence such that $\varphi_{j_k}$ converges in $L^1$ and $c_{j_k}$ converges.  So that $\tilde{\varphi}_{j_k}$ converges in $L^1$.  On the other hand,  using Proposition \ref{p3.1} will allow us to conclude that $\tilde{\varphi}_{j_k}$ converges uniformly,  which in turn implies that $\varphi_{j_k}$ converges uniformly and we can conclude that the limit function will be the viscosity solutions. 
\end{proof}

The rest of the section is devoted to the proof of Proposition \ref{p3.1}.  First we need a more refined estimate on the constant $c$ that allows one to solve $f(\lambda[\chi +dd^c\varphi])=e^{G+c}$. This kind of estimate first appeared in \cite{KNFSWW} for the complex Monge-Ampère equation. We have:
\begin{lem}\label{l3.4}
Let $G(x)$ be a smooth function on $M$.  Let $\varphi$ be the smooth solution to:
\begin{equation*}
f(\lambda[\chi+dd^c \varphi])=e^{G+c},
\end{equation*}
for some constant $c\in \bR$. Then
\begin{enumerate}
\item $c$ can be estimated from above in terms of the structural condition on $f$,  the background manifold and metric $(M,\omega_0)$,  and an upper bound on $||e^G||_{L^{2n}(\omega_0)}$.
\item $c$ can be estimated from below in terms of the function $f$,  the background metric $\omega_0$,  and a lower bound on $\int_Me^G\omega_0^n$.
\end{enumerate}
\end{lem}
\begin{proof}
First we prove the upper bound.  Let $x_0$ be the minimum point of $\varphi$ on $x_0$.  Then we can find $r_0>0$ small enough,  and a coordinate system on $B_{2r_0}(x_0)$,  such that $x_0$ corresponds to $z=0$,  and that
\begin{equation*}
g_{i\bar{j}}\ge c_0(|z|^2)_{ij}=c_0\delta_{ij}\text{ on $B_{2r_0}(x_0)$ and under this coordinate}.
\end{equation*}
We wish to apply the Alexandrov maximum principle to $\varphi+c_0|z|^2$ on $B_{r_0}(x_0)$ as follows:
\begin{equation*}
\inf_{B_{r_0}(x_0)}\big(c_0|z|^2+\varphi\big)\ge \inf_{\partial B_{r_0}(x_0)}\big(c_0|z|^2+\varphi\big)-C_nr_0\bigg(\int_{B_{r_0}(x_0)\cap \mathcal{C}_+}\det D^2\big(c_0|z|^2+\varphi\big)\bigg)^{\frac{1}{2n}}.
\end{equation*}
In the above,  $\mathcal{C}_+$ is the subset of $B_{2r_0}(x_0)$ such that $D^2(c_0|z|^2+\varphi)\ge 0$.

We can estimate further from the above.  On the set $\mathcal{C}_+$:
\begin{equation*}
\begin{split}
&\det D^2\big(c_0|z|^2+\varphi)\le 2^{2n}\big(\det(c_0|z|^2+\varphi)_{i\bar{j}}\big)^2\le 2^{2n}\big(\det(g_{i\bar{j}}+\varphi_{i\bar{j}})\big)^2\\
&\le C\big(f(\lambda[\chi+dd^c\varphi])\big)^{2n}=Ce^{2nG+2nc}.
\end{split}
\end{equation*}
In the second inequality above,  we used that $g_{i\bar{j}}\ge (c_0|z|^2)_{i\bar{j}}$ on $B_{2r_0}(x_0)$.  In the third inequality above,  we used the determinant domination condition (item (4) of Assumption \ref{a1.1N}).

On the other hand,  since $x_0$ is the global minimum point of $\varphi$,  we get
\begin{equation*}
\inf_{B_{r_0}(x_0)}(c_0|z|^2+\varphi)-\inf_{\partial B_{r_0}(x_0)}(c_0|z|^2+\varphi)\le -c_0r_0^2.
\end{equation*}
Therefore we get that:
\begin{equation*}
-c_0r_0^2\ge -C_nr_0\bigg(\int_{B_{r_0(x_0)}}Ce^{2nG+2nc}\bigg)^{\frac{1}{2n}}.
\end{equation*}
This gives a lower bound of $c$ with the said dependence.

Next we estimate the $c$ from above.  From concavity of $f$,  we see that:
\begin{equation}\label{3.3}
e^{G+c}=f(\lambda[\chi+dd^c\varphi])\le f(\lambda[\chi])+\frac{\partial F}{h_{i\bar{j}}}(\chi)\varphi_{i\bar{j}}.
\end{equation}
In the above,  one can use (\ref{2.2NNN}) to see that:
\begin{equation*}
\frac{\partial F}{\partial h_{i\bar{j}}}(\chi)\varphi_{i\bar{j}}\frac{\omega_0^n}{n!}=\det g_{i\bar{j}}\det(\frac{\partial F}{\partial h_{i\bar{j}}}(\chi))\Omega^{n-1}\wedge dd^c\varphi.
\end{equation*}
Denote $\tilde{\Omega}=\big(\det g_{i\bar{j}}\det(\frac{\partial F}{\partial h_{i\bar{j}}}(\chi))\big)^{\frac{1}{n-1}}\Omega$ and let $v_0$ be the Gauduchon factor of $\tilde{\Omega}$,  namely $dd^c(e^{(n-1)v_0}\tilde{\Omega}^{n-1})=0$.  Hence from (\ref{3.3}),  we see that:
\begin{equation*}
\begin{split}
\int_Me^{G+(n-1)v_0}e^c\frac{\omega_0^n}{n!}&\le \int_Me^{(n-1)v_0}f(\lambda[\chi])\frac{\omega_0^n}{n!}+\int_Me^{(n-1)v_0}\frac{\tilde{\Omega}^{n-1}}{(n-1)!}\wedge dd^c\varphi\\
&=\int_Me^{(n-1)v_0}f(\lambda[\chi])\frac{\omega_0^n}{n!}.
\end{split}
\end{equation*}
This gives an upper bound of $c$ with the said dependence.
\end{proof}

Now let $0<\delta<1$,  let $h_j:\bR\rightarrow \bR_{>0}$ be a sequence of smooth functions such that $h_j\ge \max(0,x)$,  $h_j\rightarrow \max(0,x)$ as $j\rightarrow \infty$.  Let $s>0$, $\kappa>1$ to be determined,  we put:
\begin{equation}\label{3.4NNNN}
A_{\delta,s,\kappa,j}=\bigg(\int_M h_j\big((1-\delta)v-\varphi-s\big)^{\kappa}e^{n\kappa G_0}\omega_0^n\bigg)^{\frac{1}{\kappa}}.
\end{equation}
Here $G_0$ is the right hand side appearing in Proposition \ref{p3.1}.  
Let $\psi_{\delta,s,j,\kappa}$ be the solution to the following problem:
\begin{equation}
\big(\omega_0+\sqrt{-1}\partial\bar{\partial}\psi_{\delta,s,j,\kappa}\big)^n=\frac{h_j\big((1-\delta)v-\varphi-s\big)}{A_{\delta,s,j,\kappa}}e^{nG_0+b_{\delta,s,j,\kappa}},\,\,\sup_M\psi_{\delta,s,j,\kappa}=0.
\end{equation}
Note that the function on the right hand side $\frac{h_j\big((1-\delta)v-\varphi-s\big)}{A_{\delta,s,j,\kappa}}e^{nG_0}$ is uniformly bounded in $L^{\kappa}$,  we may
apply Lemma 5.9 of \cite{KN2} to get:
\begin{equation*}
b_{\delta,s,j,\kappa}\ge -C,
\end{equation*}
We can also get the same result from Lemma \ref{l3.4} if $\kappa\ge 2$.
Here $C$ depends only on the background manifold and metric.  We have the following Moser-Trudinger type inequality:
\begin{lem}\label{l3.5}
There exists $c_1>0$ small enough,  $C_2>0$ large enough,  both dependent only on the structural constant of $f$,  the background metric $\omega_0$ and the choice of $\kappa>1$,  such that on the set $\{(1-\delta)v-\varphi-s>0\}$,  one has
\begin{equation}\label{3.6NN}
c_1A_{\delta,s,j,\kappa}^{-\frac{1}{n}}\big((1-\delta)v-\varphi-s\big)^{\frac{n+1}{n}}\le -\psi_{\delta,s,j,\kappa}+C_2A_{\delta,s,j,\kappa}\delta^{-(n+1)}.
\end{equation}
\end{lem}
\begin{proof}
We define,  with $\eps_0>0,\,\Lambda>0$ to be determined:
\begin{equation*}
\Phi=\eps_0\big((1-\delta)v-\varphi-s\big)-(-\psi_{\delta,s,j,\kappa}+\Lambda)^{\frac{n}{n+1}}.
\end{equation*}
In the following,  we will simply denote $\psi_{\delta,s,j,\kappa}$ by $\psi$ for simplicity.
The linearized operator is given by $L=\frac{\partial F}{\partial h_{j\bar{k}}}(\chi+dd^c \varphi)\partial_{j\bar{k}}$ and we may compute:
\begin{equation*}
\partial_{j\bar{k}}\Phi=\eps_0\big((1-\delta)v_{j\bar{k}}-\varphi_{j\bar{k}}\big)+\frac{n}{n+1}(-\psi+\Lambda)^{-\frac{1}{n+1}}\psi_{j\bar{k}}+\frac{n}{(n+1)^2}(-\psi+\Lambda)^{-\frac{n+2}{n+1}}\psi_j\psi_{\bar{k}}.
\end{equation*}
This implies:
\begin{equation*}
\begin{split}
&dd^c\Phi=\eps_0\big((1-\delta)dd^cv-dd^c\varphi\big)+\frac{n}{n+1}(-\psi+\Lambda)^{-\frac{1}{n+1}}dd^c\psi+\frac{n}{(n+1)^2}(-\psi+\Lambda)^{-\frac{n+2}{n+1}}d\psi\wedge d^c\psi\\
&\ge \eps_0\big((1-\delta)dd^cv-dd^c\varphi\big)+\frac{n}{n+1}(-\psi+\Lambda)^{-\frac{1}{n+1}}dd^c\psi.
\end{split}
\end{equation*}
Therefore:
\begin{equation}\label{3.7N}
\begin{split}
&L\Phi\ge \eps_0(1-\delta)Lv-\eps_0L\varphi+\frac{n}{n+1}(-\psi+\Lambda)^{-\frac{1}{n+1}}L\psi.\\
&\ge\eps_0(1-\delta)\frac{\partial F}{\partial h_{j\bar{k}}}(\chi+dd^c \varphi)(\chi_{j\bar{k}}+v_{j\bar{k}})-\eps_0\frac{\partial F}{\partial h_{j\bar{k}}}(\chi+dd^c \varphi)(\chi_{j\bar{k}}+\varphi_{j\bar{k}})\\
&+\frac{n}{n+1}(-\psi+\Lambda)^{-\frac{1}{n+1}}\frac{\partial F}{\partial h_{j\bar{k}}}(\chi+dd^c \varphi)(g_{j\bar{k}}+\psi_{j\bar{k}})+\eps_0\delta\frac{\partial F}{\partial h_{j\bar{k}}}(\chi+dd^c \varphi)\chi_{j\bar{k}}\\
&-\frac{n}{n+1}(-\psi+\Lambda)^{-\frac{1}{n+1}}\frac{\partial F}{\partial h_{j\bar{k}}}(\chi+dd^c \varphi)g_{j\bar{k}}.
\end{split}
\end{equation}
Now we use that $\chi\in \Gamma_{\omega_0}$,  so that there exists $c_*>0$,  such that one can write $\chi=c_*\omega_0+\hat{\chi}$ with $\lambda(\hat{\chi})\in \Gamma$.  With this observation,  we see that:
\begin{equation}\label{3.8N}
\begin{split}
&\eps_0\delta\frac{\partial F}{\partial h_{j\bar{k}}}(\chi+dd^c \varphi)\chi_{j\bar{k}}-\frac{n}{n+1}(-\psi+\Lambda)^{-\frac{1}{n+1}}\frac{\partial F}{\partial h_{j\bar{k}}}(\chi+dd^c \varphi)g_{j\bar{k}}\\
&\ge \eps_0\delta\frac{\partial F}{\partial h_{j\bar{k}}}(\chi+dd^c \varphi)\hat{\chi}_{j\bar{k}}+\frac{\partial F}{\partial h_{j\bar{k}}}(\chi+dd^c \varphi)(\eps_0\delta c_*-\frac{n}{n+1}\Lambda^{-\frac{1}{n+1}}\big)g_{j\bar{k}}.
\end{split}
\end{equation}
Using the Lemma \ref{l3.6N} below,  we may conclude that:
\begin{equation}\label{3.9N}
\frac{\partial F}{\partial h_{j\bar{k}}}(\chi+dd^c \varphi)(\chi_{j\bar{k}}+v_{j\bar{k}})\ge 0,\,\,\frac{\partial F}{\partial h_{j\bar{k}}}(\chi+dd^c \varphi)\hat{\chi}_{j\bar{k}}\ge 0.
\end{equation}
Combining (\ref{3.7N})-(\ref{3.9N}),  we see that:
\begin{equation*}
\begin{split}
&L\Phi\ge -\eps_0\frac{\partial F}{\partial h_{j\bar{k}}}(\chi+dd^c \varphi)(\chi_{j\bar{k}}+\varphi_{j\bar{k}})+\frac{n}{n+1}(-\psi+\Lambda)^{-\frac{1}{n+1}}\frac{\partial F}{\partial h_{j\bar{k}}}(\chi+dd^c \varphi)(g_{j\bar{k}}+\psi_{j\bar{k}})\\
&+(\eps_0\delta c_*-\frac{n}{n+1}\Lambda^{-\frac{1}{n+1}})\frac{\partial F}{\partial h_{j\bar{k}}}(\chi+dd^c \varphi)g_{j\bar{k}}.
\end{split}
\end{equation*}
Next we use that $f$ has homogeneity one to see that:
\begin{equation*}
\frac{\partial F}{\partial h_{j\bar{k}}}(\chi+dd^c \varphi)(\chi+dd^c \varphi)_{j\bar{k}}=\sum_i\frac{\partial f}{\partial \lambda_i}\lambda_i=f.
\end{equation*}

Also we note that:
\begin{equation*}
\begin{split}
&\frac{\partial F}{\partial h_{j\bar{k}}}(\chi+dd^c \varphi)(g_{j\bar k}+\psi_{j\bar k})\ge n\big(\det \frac{\partial F}{\partial h_{i\bar j}}\cdot \det(g_{i\bar j}+ \psi_{i\bar j})\big)^{\frac{1}{n}}\\
&\ge c_1A_{\delta,s,j,\kappa}^{-\frac{1}{n}}\big((1-\delta)-\varphi-s\big)_+^{\frac{1}{n}}F(\chi+dd^c\varphi)e^{\frac{b_{\delta,s,j,\kappa}}{n}}.
\end{split}
\end{equation*}

The first inequality above follows from arithmetic-geometric inequality,  which we prove in more detail in Lemma \ref{l4.8New}.   
The second inequality above used that:
\begin{equation*}
\det(g_{i\bar{j}}+\psi_{i\bar{j}})=\frac{h_j((1-\delta)v-\varphi-s)}{A_{\delta,s,j,\kappa}}e^{nG_0+b_{\delta,s,j,\kappa}}\ge \frac{\big((1-\delta)v-\varphi-s\big)_+}{A_{\delta,s,j,\kappa}}e^{nG_0+b_{\delta,s,j,\kappa}}.
\end{equation*}

To proceed further,  assume that $\Phi$ achieves positive maximum at $p\in M$.  We obtain the following when evaluated at $p$:

Therefore,  when evaluated at $p$,  we get:
\begin{equation}\label{3.6}
\begin{split}
0&\ge L\Phi\ge -\eps_0f+c_2A_{\delta,s,j,\kappa}^{-\frac{1}{n}}\big((1-\delta)-\varphi-s\big)_+^{\frac{1}{n}}(-\psi+\Lambda)^{-\frac{1}{n+1}}f\\
&+(\eps_0\delta c_*-\frac{n}{n+1}\Lambda^{-\frac{1}{n+1}})\frac{\partial F}{\partial h_{j\bar{k}}}g_{j\bar{k}}.
\end{split}
\end{equation}

In the above,  $c_2$ depends only on the structural constants of $f$ and also the background metric.  Here we used that $b_{\delta,s,j,\kappa}$ has a universal lower bound (except depending on the choice of $\kappa>1$).
Now we can choose the parameter $\eps_0$ first so as to have:
\begin{equation}\label{4.11New}
-2\eps_0+c_2A_{\delta,s,j,\kappa}^{-\frac{1}{n}}\eps_0^{-\frac{1}{n}}=0.
\end{equation}
Now we fix this choice of $\eps_0$ and then choose $\Lambda$ so that:
\begin{equation*}
\frac{1}{2}\eps_0\delta c_*=\frac{n}{n+1}\Lambda^{-\frac{1}{n+1}}. 
\end{equation*}
With this choice,  we are going to get a contradiction from (\ref{3.6}).  Indeed,  since $\Phi(p)>0$,  we see that at $p$,  one has
\begin{equation*}
((1-\delta)v-\varphi-s)^{\frac{1}{n}}(-\psi+\Lambda)^{-\frac{1}{n+1}}>\eps_0^{-\frac{1}{n}}.
\end{equation*}
Then we see from (\ref{3.6}) that:
\begin{equation*}
\begin{split}
&L\Phi\ge -\eps_0f+c_2A_{\delta,s,j,\kappa}^{-\frac{1}{n}}\big((1-\delta)-\varphi-s\big)_+^{\frac{1}{n}}(-\psi+\Lambda)^{-\frac{1}{n+1}}f\ge -\eps_0f +c_2A_{\delta,s,j,k}^{-\frac{1}{n}}\eps_0^{-\frac{1}{n}}f\\
&=-\eps_0f+2\eps_0f>0.
\end{split}
\end{equation*}
The equality on the second line used the choice of $\eps_0$ specified in (\ref{4.11New}).  This contradicts with (\ref{3.6}) because we had $L\Phi\le 0$ in (\ref{3.6}).
\end{proof}
In the above proof,  we used the following lemma to get to (\ref{3.9N}).
\begin{lem}\label{l3.6N}
Let $A$,  $B$ be two Hermitian matrices such that $\lambda(g^{i\bar{k}}A_{j\bar{k}}),\,\lambda(g^{i\bar{k}}B_{j\bar{k}})\in \Gamma$.  Define $F(h)=f\big(\lambda(g^{i\bar{k}}h_{j\bar{k}})\big)$.  Then we have:
\begin{equation*}
\frac{\partial F}{\partial h_{j\bar{k}}}(A)B_{j\bar{k}}\ge 0.
\end{equation*}
\end{lem}
\begin{proof}
We first prove this statement,  assuming that $g_{i\bar{j}}=\delta_{ij}$ and that $A$ is diagonal.  With this assumption,  we then have:
\begin{equation*}
\frac{\partial F}{\partial h_{j\bar{k}}}(A)=\frac{\partial f}{\partial \lambda_j}(\lambda(A))\delta_{jk}.
\end{equation*}
Without loss of generality,  we may assume that $\lambda_1\le \lambda_2\le \cdots \lambda_n$,  so that $\frac{\partial f}{\partial \lambda_1}\ge \frac{\partial f}{\partial \lambda_2} \ge \cdots \ge \frac{\partial f}{\partial \lambda_n}$.  Therefore,  from above we see that
\begin{equation*}
\frac{\partial F}{\partial h_{j\bar{k}}}(A)B_{j\bar{k}}=\frac{\partial f}{\partial \lambda_j}(\lambda(A))B_{j\bar{j}}\ge \frac{\partial f}{\partial \lambda_j}(\lambda(A))\mu_j.
\end{equation*}
In the above,  $\mu_j$ are the eigenvalues of $B$,  also listed in the increasing order.  The inequality above used the Horn-Shur lemma,  which says that the vector $(B_{1\bar{1}},\cdots,B_{n\bar{n}})$ is contained in the convex hull of $(\mu_{\sigma(1)},\cdots,\mu_{\sigma(n)})$,  where $\sigma$ is a permutation of the indices.  Moreover,  since $\frac{\partial f}{\partial \lambda_j}$ is in the decreasing order,  $\frac{\partial f}{\partial \lambda_j}\mu_j$ will be minimized if $\mu_j$ is in the increasing order,  hence the inequality above.  Now it only remains to show $\frac{\partial f}{\partial \lambda_j}(\lambda(A))\mu_j \ge 0$.  Indeed,  we just need to note that $\frac{\partial f}{\partial \lambda_j}(\lambda(A))\mu_j=\frac{d}{dt}|_{t=0}f(\lambda(A)+t\mu)$.  Note that $t\mapsto f(\lambda(A)+t\mu)$ is concave,  and is bounded from below on $[0,\infty)$,  hence $\frac{d}{dt}|_{t=0}f(\lambda(A)+t\mu)\ge 0$.

Next we explain how to reduce the general case to the special case considered above.

First we observe that one can reduce to the case when $g_{i\bar{j}}=\delta_{ij}$.  Indeed,  we may assume that there is an invertible $n\times n$ matrix,  such that $g=P\overline{P^T}$ (that is,  $g_{i\bar{j}}=P_{ir}\overline{P_{jr}}$).
On the other hand,
\begin{equation*}
(g^{i\bar{k}}A_{j\bar{k}})^i_j=\big((g^{-1})_{ki}A_{j\bar{k}}\big)^i_j=(A\cdot g^{-1})_{ij}=(A\cdot \overline{(P^T)^{-1}}\cdot P^{-1})_{ij}
\end{equation*}
Therefore,  if we define $\tilde{A}=P^{-1}A\overline{(P^T)^{-1}}$,  then one has $\lambda((g^{i\bar{k}}A_{j\bar{k}})_j^i)=\lambda(\tilde{A})$.  Likewise,  we define $\tilde{B}=P^{-1}B\overline{(P^T)^{-1}}$,  then one gets $\lambda(\tilde{A}),\,\lambda(\tilde{B})\in \Gamma$.  Also we define $\tilde{F}(h)=f\big(\lambda(h)\big)$, where $h$ is a Hermitian matrix and $\lambda(h)$ means the usual eigenvalue of $h$,  then we have $F(h)=\tilde{F}\big(P^{-1}h\overline{(P^T)^{-1}}\big)$.  From this we may calculate:
\begin{equation*}
\begin{split}
&\frac{\partial F}{\partial h_{j\bar{k}}}(A)B_{j\bar{k}}=\frac{\partial \tilde{F}}{\partial h_{p\bar{q}}}(\tilde{A})\frac{\partial (P^{-1}h\overline{(P^T)^{-1}})_{p\bar{q}}}{\partial h_{j\bar{k}}}P_{ja}\tilde{B}_{a\bar{b}}\overline{(P^T)_{ak}}\\
&=\frac{\partial \tilde{F}}{\partial h_{p\bar{q}}}(\tilde{A})(P^{-1})_{pj}\overline{(P^{-1})_{qk}}P_{j\bar{a}}\tilde{B}_{a\bar{b}}\overline{(P^T)_{ak}}=\frac{\partial \tilde{F}}{\partial h_{a\bar{b}}}(\tilde{A})\tilde{B}_{a\bar{b}}.
\end{split}
\end{equation*}
Therefore,  we see that,  as long as we can verify the lemma with $\tilde{A},\,\tilde{B},\,\tilde{F}$ (which is equivalent to taking $g=I$),  the general statement would follow.

Next,  we explain why we can assume that $A$ is diagonal.   Indeed,  from the definition of $F$,  one has $F(h)=F(Uh\overline{U^T})$,  for any unitary matrix $U$.  Therefore one has
\begin{equation*}
\frac{\partial F}{\partial h_{i\bar{j}}}(h)=\frac{\partial F}{\partial h_{p\bar{q}}}(Uh\overline{U^T})U_{pi}\overline{U_{qj}}.
\end{equation*}
Now we choose $h=A$,  and choose $U$ to be the unitary matrix such that $UA\overline{U^T}$ is diagonal.  Then one gets:
\begin{equation*}
\frac{\partial F}{\partial h_{i\bar{j}}}(A)B_{i\bar{j}}=\frac{\partial F}{\partial h_{p\bar{q}}}(UA\overline{U^T})U_{pi}\overline{U_{qj}}B_{i\bar{j}}.
\end{equation*}
All we need to do is to consider $\tilde{B}=UB\overline{U^T}$ in place of $B$,  and they have the same eigenvalues.
\end{proof}
We also used the following lemma:
\begin{lem}\label{l4.8New}
Let $A,\,B$ be two positive definite Hermitian matrices.  Then we have:
\begin{equation*}
\frac{1}{n}tr\big(\bar{A}^TB\big)\ge \big(\det A\cdot \det B\big)^{\frac{1}{n}}.
\end{equation*}
\end{lem}
\begin{proof}
Since $B$ is positive definite and Hermitian,  we can find an invertible matrix $P$ such that $B=P\bar{P}^T$.  Then we have:
\begin{equation*}
\frac{1}{n}tr\big(\bar{A}^TB\big)=\frac{1}{n}tr\big(\bar{A}^TP\bar{P}^T\big)=\frac{1}{n}tr\big(\bar{P}^T\bar{A}^TP\big)\ge \big(\det(\bar{P}^T\bar{A}^TP)\big)^{\frac{1}{n}}=\big(\det A\cdot \det B\big)^{\frac{1}{n}}.
\end{equation*}
The inequality above follows from the Arithmetic-Geometric inequalities applied to the eigenvalues of $\bar{P}^T\bar{A}^TP$.
\end{proof}

As a consequence of Lemma \ref{l3.5},  we see that:
\begin{cor}\label{c3.6}
Denote 
\begin{equation}\label{3.8NN}
A_{\delta,s,\kappa}=\bigg(\int_M\big((1-\delta)v-\varphi-s\big)_+^{\kappa}e^{\kappa nG_0}\omega_0^n\bigg)^{\frac{1}{\kappa}}.
\end{equation}
There exists $\beta_0>0$,  $C>0$ which depends only on the background manifold and metric,  as well as the choice of $\kappa>1$,  such that 
\begin{equation*}
\int_M\exp\big(\beta_0A_{\delta,s,\kappa}^{-\frac{1}{n}}\big((1-\delta)v-\varphi-s\big)_+^{\frac{n+1}{n}}\big)\omega_0^n\le \exp\big(CA_{\delta,s,\kappa}\delta^{-(n+1)}\big).
\end{equation*}
\end{cor}
\begin{proof}
Using the following lemma,  we see that there exists $\alpha_0>0$,  $C>0$ such that for any $\psi\in PHS(M,\omega_0)$ with $\sup_M\psi=0$,  one has:
\begin{equation*}
\int_Me^{-\alpha_0\psi}\omega_0^n\le C.
\end{equation*}
Now one multiplies both sides of (\ref{3.6NN}) with $\alpha_0$,  raise to the exponential,  and integrate on $M$,  we get that 
\begin{equation*}
\int_M\exp\big(\beta_0A_{\delta,s,j,\kappa}^{-\frac{1}{n}}\big((1-\delta)v-\varphi-s\big)_+^{\frac{n+1}{n}}\big)\omega_0^n\le \exp\big(CA_{\delta,s,j,\kappa}\delta^{-(n+1)}\big).
\end{equation*}
Note that from the definition of $A_{\delta,s,j,\kappa}$,  we see that $A_{\delta,s,j,\kappa}\rightarrow A_{\delta,s,\kappa}$ in (\ref{3.4NNNN}) as $j\rightarrow \infty$.
\end{proof}
In the above,  we used the following lemma whose proof may be found in \cite{Ti}:
\begin{lem}
Let $(M,\omega_0)$ be a compact Hermitian manifold.  Then there exists $\alpha_0>0,\,C>0$,  such that for any $\psi\in PSH(M,\omega_0)$ with $\sup_M\psi=0$,  one has that
\begin{equation*}
\int_Me^{-\alpha_0\psi}\omega_0^n\le C.
\end{equation*}
\end{lem}

Using the above estimate,  we can get the pointwise upper bound of $(1-\delta)v-\varphi-s$.  For this we have:
\begin{lem}\label{l3.8}
Assume that $\delta>0,\,s_0>0$ are chosen so that $A_{\delta,s_0,\kappa}\le \delta^{n+1}$.  Assume that $e^{nG_0}\in L^{p_0}(M,\omega_0^n)$ for some $p_0>1$.  Define $\Omega_{\delta,s}=\{x\in M:(1-\delta)v-\varphi-s>0\}$ and put $u(s)=\int_{\Omega_{\delta,s}}e^{nG_0}\omega_0^n$.  Then for any $0<\delta_*<\frac{1}{n}$,  we may choose $\kappa>1$ sufficiently close to 1,  such that there exists $C_*>0$,  depending on $\delta_*$,  $||e^{nG_0}||_{L^{p_0}}$,  the background metric and the choice of $\kappa$,  such that
\begin{equation*}
tu(s+t)\le C_*u(s)^{1+\delta_*},
\end{equation*}
for any $s\ge s_0$,  $t>0$.  The above choice of $\kappa$ depends only on $\delta_*,\,n,\,p_0$.
\end{lem}
\begin{proof}
Let $s\ge s_0$,  so that $A_{\delta,s,\kappa}\le A_{\delta,s_0,\kappa}\le \delta^{n+1}$.  
We then obtain from Corollary \ref{c3.6} that for any positive integer $q$,  one has that:
\begin{equation}\label{3.8}
\int_M\big((1-\delta)v-\varphi-s\big)_+^{\frac{n+1}{n}q}\omega_0^n\le C(q)A_{\delta,s}^{\frac{q}{n}}.
\end{equation}
Therefore,
\begin{equation}\label{3.9}
\begin{split}
&A_{\delta,s,\kappa}=||((1-\delta)v-\varphi-s)e^{nG_0}||_{L^{\kappa}(\Omega_{\delta,s})}\\
&\le ||((1-\delta)v-\varphi-s)||_{L^{\frac{n+1}{n}q}(\Omega_{\delta,s})}||e^{nG_0}||_{L^{q'}(\Omega_{\delta,s})}\\
&\le C(q)A_{\delta,s,\kappa}^{\frac{1}{n+1}}||e^{nG_0}||_{L^{q'}(\Omega_{\delta,s})}\le C(q)A_{\delta,s,\kappa}^{\frac{1}{n+1}}||e^{nG_0}||^{\lambda}_{L^1(\Omega_{\delta,s})}||e^{nG_0}||^{1-\lambda}_{L^{p_0}}.
\end{split}
\end{equation}
In the second line above,  the $q$ will be chosen sufficiently large and $q'$ is such that $\frac{1}{\kappa}=\frac{n}{n+1}\frac{1}{q}+\frac{1}{q'}$.  By choose $\kappa>1$ sufficiently close to 1 and $q$ sufficiently large,  we may make $q'>1$ arbitrarily close to 1.  This just follows from H\"older's inequality. 

 In the first inequality of the third line,  we used (\ref{3.8}).  In the second inequality of the third line,  $0<\lambda<1$ satisfies: $\frac{1}{q'}=\lambda+\frac{1-\lambda}{p_0}$.  By making $q'>1$ sufficiently close to 1,  we may make $\lambda$ as close to 1 as we want.

Hence we see from (\ref{3.9}) that,  for any $\eps>0$:
\begin{equation*}
A_{\delta,s,\kappa}\le C_{\eps}||e^{nG_0}||_{L^1(\Omega_{\delta,s})}^{\frac{n+1}{n}-\eps}.
\end{equation*}
Here $C_{\eps}$ above depends on $||e^{nG_0}||_{L^{p_0}}$ as well.  

On the other hand,  since $(1-\delta)v-\varphi-s>t$ on $\Omega_{s+t}$,  one has:
\begin{equation*}
A_{\delta,s,\kappa}\ge t||e^{nG_0}||_{L^{\kappa}(\Omega_{\delta,s+t})}\ge tc_1||e^{nG_0}||_{L^1(\Omega_{\delta,s+t})}.
\end{equation*}
Hence the result follows.
\end{proof}

We need to use the following Lemma of De Giorgi which was first used in the setting of complex Monge-Ampère equations in \cite{K}:
\begin{lem}\label{l3.9}
Let $\phi:[0,\infty)\rightarrow [0,\infty)$ be a decreasing function,  such that there exists $\mu>0$,  $B_0>0$,  $s_0\ge 0$,  such that for any $r>0,\,s\ge s_0$,  one has:
\begin{equation*}
r\phi(s+r)\le B_0\phi(s)^{1+\mu}.
\end{equation*}
Then $\phi(s)\equiv 0$ for $s\ge s_0+\frac{2B_0\phi(s_0)^{\mu}}{1-2^{-\mu}}.$
\end{lem}
\begin{proof}
We can choose a sequence $\{s_k\}_{k\ge 1}$ by induction:
\begin{equation*}
s_{k+1}-s_k=2B_0\phi(s_k)^{\mu}.
\end{equation*}
Then we choose $s=s_k,\,r=s_{k+1}-s_k$,  we see that
\begin{equation*}
\phi(s_{k+1})\le \frac{B_0\phi(s_k)^{1+\mu}}{s_{k+1}-s_k}\le \frac{1}{2}\phi(s_k).
\end{equation*}
That is,  $\phi(s_k)\le 2^{-k}\phi(s_0)$,  hence:
\begin{equation*}
s_{k+1}-s_k\le 2B_0\phi(s_0)^{\mu}2^{-k\mu}.
\end{equation*}
Therefore,
\begin{equation*}
\sum_{k=0}^{\infty}(s_{k+1}-s_k)\le \sum_{k=0}^{\infty}2B_0\phi(s_0)^{\mu}2^{-k\mu}=\frac{2B_0\phi(s_0)^{\mu}}{1-2^{-\mu}}.
\end{equation*}
It implies that $s_k$ is increasing and bounded from above.  Hence we see that $s_k\rightarrow s_{\infty}$.  Moreover,  for all $s\ge s_{\infty}$,  we see that:
\begin{equation*}
\phi(s)\le \phi(s_{\infty})\le \phi(s_k)\le 2^{-k}\phi(s_0).
\end{equation*}
Letting $k\rightarrow\infty$ we see that $\phi(s)\equiv 0$ for $s\ge s_{\infty}$.  Moreover,  one can find that 
$s_{\infty}\le s_0+\frac{2B_0\phi(s_0)^{\mu}}{1-2^{-\mu}}$.
\end{proof}

Combining Lemma \ref{l3.8} and \ref{l3.9},  we see that
\begin{lem}\label{l3.10}
Let $0<\delta<1$,  $s_0>0$ be chosen so that $A_{\delta,s_0,\kappa}\le \delta^{n+1}$.  Then for any $0<\nu<\frac{1}{n}(1-\frac{1}{p_0})$,  we may choose $\kappa>1$ sufficiently close to 1,  such that there exists $C>0$,  depending only on $\nu$,  $||e^{nG_0}||_{L^{p_0}}$,  the background metric and the choice of $\kappa$  such that
\begin{equation*}
\sup_M\big((1-\delta)v-\varphi\big)\le s_0+Cvol\big(\Omega_{\delta,s_0}\big)^{\nu}.
\end{equation*}
\end{lem}
\begin{proof}
Combining Lemma \ref{l3.8} and \ref{l3.9},  and keep in mind that $u(s)=\int_{\Omega_{\delta,s}}e^{nG_0}\omega_0^n$,  we get: for any $0<\delta_*<\frac{1}{n}$:
\begin{equation}\label{3.11N}
\sup_M((1-\delta)v-\varphi)\le s_0+C_*\big(\int_{\Omega_{\delta,s_0}}e^{nG_0}\omega_0^n\big)^{\delta_*}.
\end{equation}
Here $C$ depends only on the choice of $\delta_*$,  $||e^{nG_0}||_{L^{p_0}}$,  the background metric and the choice of $\kappa>1$.  
We then apply H\"older's inequality to see that:
\begin{equation*}
\int_{\Omega_{\delta,s_0}}e^{nG_0}\omega_0^n\le ||e^{nG_0}||_{L^{p_0}}vol(\Omega_{\delta,s_0})^{1-\frac{1}{p_0}}.
\end{equation*}
Plugging this to (\ref{3.11N}) gives us the result.
\end{proof}
With a little more work,  we wish to get rid of the $\delta$ in the above estimate:
\begin{lem}\label{l3.11}
Assume that $0<\delta<1$,  and $s_0>0$ are chosen so that:
\begin{enumerate}
\item $s_0\ge 2\delta||v||_{L^{\infty}}$,
\item $A_{\delta,s_0,\kappa}\le \delta^{n+1}$.
\end{enumerate}
Then for any $0<\nu<\frac{1}{n}(1-\frac{1}{p_0})$,  there exists $C>0$,  depending only on $\nu$,  $||e^{nG_0}||_{L^{p_0}}$,  and the background metric,  such that
\begin{equation*}
\sup_M(v-\varphi)\le \frac{3s_0}{2}+Cs_0^{-\nu}||(v-\varphi)_+||_{L^1(\omega_0^n)}^{\nu}.
\end{equation*}
\end{lem}
\begin{proof}
Let $s_0>0,\,0<\delta<1$ be as stated in the lemma.  Then we have
\begin{equation*}
vol(\Omega_{\delta,s_0})\le \frac{1}{s_0}\int_{\Omega_{\delta,s_0}}\big((1-\delta)v-\varphi\big)_+\omega_0^n\le \frac{1}{s_0}\big(||(v-\varphi)_+||_{L^1}+\delta||v||_{L^{\infty}}vol(\Omega_{\delta,s_0})\big).
\end{equation*}
Since $s_0\ge 2\delta||v||_{L^{\infty}}$,  one gets that:
\begin{equation*}
vol(\Omega_{\delta,s_0})\le \frac{2}{s_0}||(v-\varphi)_+||_{L^1}.
\end{equation*}
Therefore,  using Lemma \ref{l3.10},  we get:
\begin{equation*}
\sup_M(v-\varphi)\le \sup_M((1-\delta)v-\varphi)+\delta||v||_{L^{\infty}}\le \frac{3s_0}{2}+Cvol(\Omega_{\delta,s_0})^{\nu}\le \frac{3s_0}{2}+C2^{\nu}s_0^{-\nu}||(v-\varphi)_+||_{L^1}^{\nu}.
\end{equation*}
So the result follows.
\end{proof}
For $0<\delta<1$,  we define $s_0(\delta)$ to be the minimum of $s_0$ that satisfies $s_0\ge 2\delta||v||_{L^{\infty}}$ and $A_{\delta,s_0,\kappa}\ge \delta^{n+1}$.  

At this point,  it only remains to estimate $s_0(\delta)$,  and we have:
\begin{lem}\label{l3.12}
For any $0<\delta<1$,  we define $s_0(\delta)$ to be the smallest $s_0$ such that $s_0\ge 2\delta||v||_{L^{\infty}}$ and $A_{\delta,s_0,\kappa}\le \delta^{n+1}$.  Then for any $\mu>\frac{np_0}{p_0-1}$,  there exists $C_{\mu}$,  depending only on $\mu$,  $||e^{nG_0}||_{L^{p_0}}$,  the background metric and the choice of $\kappa$,  such that
\begin{equation*}
s_0(\delta)\le \max\big(2\delta||v||_{L^{\infty}},C_{\mu}\delta^{-\mu}||(v-\varphi)_+||_{L^1}\big).
\end{equation*}
\end{lem}
\begin{proof}
Note that from (\ref{3.8NN}),  $A_{\delta,s,\kappa}$ depends continuously on $s$.  Therefore,  with $s_0=s_0(\delta)$,  one must have either $s_0=2\delta||v||_{L^{\infty}}$ or $A_{\delta,s_0,\kappa}=\delta^{n+1}$.  If the first possibility happens,  then we are done.  Now we look at what happens when $A_{\delta,s_0}=\delta^{n+1}$.  For this,  let $\beta>\kappa$,  we can then calculate:
\begin{equation*}
\begin{split}
&A_{\delta,s_0,\kappa}=\big(\int_{\Omega_{\delta,s_0}}\big((1-\delta)v-\varphi-s\big)_+^{\kappa}e^{\kappa nG_0}\omega_0^n\big)^{\frac{1}{\kappa}}\le ||((1-\delta)v-\varphi-s)_+||_{L^{\frac{\kappa\beta}{\beta-\kappa}}(\omega_0^n)}||e^{nG_0}||_{L^{\beta}(\Omega_{\delta,s_0})}\\
&\le C(\beta)A_{\delta,s_0,\kappa}^{\frac{1}{n+1}}\cdot ||e^{nG_0}||_{L^{p_0}}vol(\Omega_{\delta,s_0})^{\frac{1}{\beta}-\frac{1}{p_0}}\le C(\beta)A_{\delta,s_0,\kappa}^{\frac{1}{n+1}}\cdot ||e^{nG_0}||_{L^{p_0}}(2s_0^{-1})^{\frac{1}{\beta}-\frac{1}{p_0}}||(v-\varphi)_+||_{L^1}^{\frac{1}{\beta}-\frac{1}{p_0}}.
\end{split}
\end{equation*}
Using that $A_{\delta,s_0}=\delta^{n+1}$,  we get that:
\begin{equation*}
s_0\le C'(\beta)\delta^{-\frac{np_0\beta}{p_0-\beta}}||(v-\varphi)_+||_{L^1}.
\end{equation*}
By choosing $\beta>1$ as close to 1 as we want,  we may make $\frac{np_0\beta}{p_0-\beta}$ sufficiently close to $\frac{np_0}{p_0-1}$.
\end{proof}
From this,  the Proposition \ref{p3.1} immediately follows:
\begin{proof}
(Of Proposition \ref{p3.1}) 
Assume that $||(v-\varphi)_+||_{L^1}<1$ for the moment,  then for any $\mu>\frac{np_0}{p_0-1}$,  we wish to take:
\begin{equation*}
\delta=||(v-\varphi)_+||_{L^1}^{\frac{1}{\mu+1}}.
\end{equation*}
With this choice,  we see from Lemma \ref{l3.12} that:
\begin{equation*}
s_0(\delta)\le C||(v-\varphi)_+||_{L^1}^{\frac{1}{\mu+1}}.
\end{equation*}
Here $C$ depends on $\mu,\,||v||_{L^{\infty}},\,||e^{nG_0}||_{L^{p_0}}$,   the background metric and the choice of $\kappa$.  That is,  if we take $s_0=C||(v-\varphi)_+||_{L^1}^{\frac{1}{\mu+1}}$,  it will satisfy $s_0\ge 2\delta||v||_{L^{\infty}}$ and $A_{s_0,\delta}\le \delta^{n+1}$.  Hence,  we may use Lemma \ref{l3.11} to conclude that 
\begin{equation*}
\sup_M(v-\varphi)\le C\big(||(v-\varphi)_+||_{L^1}^{\frac{1}{\mu+1}}+||(v-\varphi)_+||_{L^1}^{\frac{\nu\mu}{\mu+1}}\big)\le C'||(v-\varphi)_+||_{L^1}^a.
\end{equation*}
A careful examination of the exponents shows that one can take any $a<\frac{p_0-1}{np_0+p_0-1}$.  
If $||(v-\varphi)_+||_{L^1}\ge 1$,  then the situation is trivial.  Indeed,  we have estimate for $||\varphi||_{L^{\infty}}$ from Guo-Phong's  $L^{\infty}$ estimate for Hermitian case \cite{GP}.
\end{proof}

\section{uniqueness issues}
\subsection{sup/inf convolution of the viscosity solution}
The basic strategy to prove uniqueness of viscosity solutions will be to perform sup/inf convolution of the solution,  so that we get a sub/super solution that is punctually second order differentiable a.e.  Then we get the pointwise differential inequality wherever the solution is punctually second order differentiable.

First we explain what we mean by punctually second order differentiable which is different from the usual definition of second order differentiable:
\begin{defn}\label{def 5.1}
Let $U\subset \bR^d$ be an open set,  and $x_0\in U$.  Let $\varphi$ be a function defined on $U$.  We say that $\varphi$ is punctually second order differentiable at $x_0$,  if there exists a quadratic polynomial $P_{\varphi,x_0}(x)$,  such that
\begin{equation*}
\lim_{r\rightarrow 0}r^{-2}\sup_{x\in B_r(x_0)}|\varphi(x)-P_{\varphi,x_0}(x)|=0.
\end{equation*}
\end{defn}
It would be useful to observe that:
\begin{lem}\label{l5.2new}
Let $U\subset \bR^d$ be an open set and $x_0\in U$.  Let $\varphi$ be a function defined on $U$.  Then $\varphi$ is punctually second order differentiable at $x_0$ if and only if there is a $C^2$ function $\psi$ defined in a neighborhood of $x_0$ such that 
\begin{equation}\label{5.1New}
\lim_{r\rightarrow 0}r^{-2}\sup_{x\in B_r(x_0)}|\varphi(x)-\psi(x)|=0.
\end{equation}
Moreover,  let $\psi_1,\,\psi_2$ be any two $C^2$ functions defined near $x_0$ that satisfies (\ref{5.1New}),  then the second order Taylor polynomial of $\psi_1$ and $\psi_2$ are equal.
\end{lem}
\begin{proof}
If $\varphi$ is punctually second order differentiable at $x_0$,  then we can take the $C^2$ function $\psi=P_{\varphi,x_0}$ which is defined on $\bR^d$.

Conversely,  assume that there is a $C^2$ function $\psi$ defined near $x_0$ such that \sloppy $\lim_{r\rightarrow 0}r^{-2}\sup_{x\in B_r(x_0)}|\varphi(x)-\psi(x)|=0$.  Let $P$ be the Taylor polynomial of $\psi$ up to second order,  then $\lim_{r\rightarrow 0}r^{-2}\sup_{x\in B_r(x_0)}|\psi(x)-P(x)|=0$.  Therefore,  $P$ will be the quadratic polynomial satisfying the condition in the definition of punctually second order differentiability.

Let $\psi_1,\,\psi_2$ be any two $C^2$  functions satisfying (\ref{5.1New}),  then 
$\lim_{r\rightarrow 0}r^{-2}\sup_{x\in B_r(x_0)}|\psi_1(x)-\psi_2(x)|=0$.  This implies the Taylor polynomial of $\psi_1$ and $\psi_2$ must agree up to second order.
\end{proof}
Using this lemma,  we observe that the notion of punctually second order differentiability is invariant under diffeomorphism.  Indeed,  we have:
\begin{lem}\label{l5.3}
Let $U,\,V\subset \bR^d$ be open sets.  Let $\Phi:V\rightarrow U$ be a smooth map.  Let $x_0\in U,\,y_0\in V$,  and $x_0=\Phi(y_0)$.  Let $\varphi$ be a function defined on $U$ which is punctually second order differentiable at $x_0$,  then $\varphi\circ \Phi$ is punctually second order differentiable at $y_0$.
\end{lem}
\begin{proof}
Let $\psi(x)$ be a $C^2$ function defined in a neighborhood of $x_0$ with \sloppy $\lim_{r\rightarrow 0}r^{-2}\sup_{B_r(x_0)}|\varphi(x)-\psi(x)|=0$.  Then $\psi\circ \Phi$ will be a $C^2$ function defined in a neighborhood of $y_0$.  We wish to show that:
\begin{equation*}
\lim_{r\rightarrow 0}r^{-2}\sup_{y\in B_r(y_0)}|\varphi\circ \Phi(y)-\psi\circ \Phi(y)|=0.
\end{equation*}
Note that there exists $C_1>0$ such that $\Phi(B_r(y_0))\subset B_{C_1r}(x_0)$ for all $r>0$ small enough.  Hence $\sup_{y\in B_r(y_0)}|\varphi\circ \Phi(y)-\psi\circ \Phi(y)|\le \sup_{x\in B_{C_1r}(x_0)}|\varphi(x)-\psi(x)|$.  So the result follows.
\end{proof}

If $\varphi$ is defined on a manifold,  then we can define punctually second order differentiability as follows:
\begin{defn}
    Let $M$ be a manifold and $x_0\in M$. Let $\varphi$ be a function defined on $M$. We say that $\varphi$ is punctually second order differentiable at $x_0$ if there exists a coordinate chart around $x_0$ such that $\varphi$ is punctually second order differentiable at $x_0$ under this coordinate chart in the sense of Definition \ref{def 5.1}.
\end{defn}
Because of Lemma \ref{l5.3},  we see that the notion of punctually second order differentiability is actually independent of the choice of coordinate chart.  In other words,  $\varphi$ would be punctually second order differentiable under any coordinate chart that contains $x_0$.  Moreover,  the Hessian of $\varphi$ is well-defined at $x_0$,  and one defines that to be the Hessian of any $C^2$ function $\psi$ at $x_0$,  where $\psi$ satisfies (\ref{5.1New}).
This definition is independent of the choice of $\psi$ due to Lemma \ref{l5.2new}.

One thing we observe is that,  if a viscosity subsolution/supersolution is twice differentiable at a point,  then at that point,  the differential inequality holds in the classical sense.
\begin{lem}
Let $G:M\times \bR\rightarrow \bR$ be a continuous 
function.
\begin{enumerate}
\item Let $\varphi\in C(M)$ be a viscosity subsolution to $F(\chi+dd^c \varphi)=e^{G(x,\varphi)}$.  Assume that $\varphi$ is punctually second order differentiable at $x_0$.  Then the following holds in the classical sense:
\begin{equation*}
F(\chi+dd^c\varphi)(x_0)\ge e^{G(x_0,\varphi(x_0))},
\end{equation*}
\item Let $\psi\in C(M)$ be a viscosity supersolution to $F(\chi+dd^c \psi)=e^{G(x,\psi)}$.  Assume that $\psi$ is punctually second order differentiable at $x_0$. Then one has
either
\begin{equation*}
\lambda[\chi+dd^c \psi](x_0)\notin Int(\Gamma)
\end{equation*}
or
\begin{equation*}
\lambda[\chi+dd^c \psi](x_0)\in Int(\Gamma)\text{ and }F(\chi+ dd^c \psi)(x_0)\le e^{G(x_0,\psi(x_0))}.
\end{equation*}
\end{enumerate}
\end{lem}
\begin{proof}
First we prove (1).  Let $\tilde{\varphi}$ be a $C^2$ function defined in a neighborhood of $x_0$ such that $\lim_{r\rightarrow 0}r^{-2}\sup_{B_r(x_0)}|\varphi(x)-\tilde{\varphi}(x)|=0$.
Then we see that for any $\eps>0$,  $\tilde{\varphi}(x)+\eps|x-x_0|^2$ would touch $\varphi$ from above at $x_0$ in the sense of Definition \ref{touch}. By writing $|x-x_0|^2$,  we have chosen some holomorphic coordinate chart near $x_0$.  Also we may assume that $\delta_{ij}\le Cg_{i\bar{j}}$ on this chart,  for some $c>0$.  
 Since $\varphi$ is a viscosity subsolution,  we see that:
\begin{equation*}
f\big(\lambda[\chi+ dd^c\tilde{\varphi}(x_0)+C\eps \omega_0]\big)\ge f\big(\lambda[\chi+ dd^c(P_{\varphi,x_0}+\eps|x-x_0|^2)]\big)\ge e^{G(x_0,\varphi(x_0))}.
\end{equation*}
Let $\eps\rightarrow 0$,  we see that $\lambda[\chi+dd^c\tilde{\varphi}(x_0)+C\eps \omega_0]\rightarrow \lambda[\chi+dd^c\tilde{\varphi}](x_0)$.  Since $f$ is a coutinuous function,  we see that
\begin{equation*}
f\big(\lambda[\chi+ dd^c\tilde{\varphi}]\big)(x_0)\ge e^{G(x_0,\varphi(x_0))}.
\end{equation*}
The proof of (2) is similar.  Indeed,  let $\tilde{\psi}$ be a $C^2$ function defined in a neighborhood of $x_0$ such that $\lim_{r\rightarrow 0}r^{-2}\sup_{B_r(x_0)}|\psi(x)-\tilde{\psi}(x)|=0$,  we know that for any $\eps>0$,  $\tilde{\psi}-\eps|x-x_0|^2$ would touch $\psi$ from below.

If it happens that $\lambda[\chi+dd^c \tilde{\psi}](x_0)\notin Int(\Gamma)$,  then we are in the first possibility and we are done.  

If $\lambda[\chi+dd^c \tilde{\psi}](x_0)\in Int(\Gamma)$,  then $\lambda[\chi+dd^c\tilde{\psi}(x_0)-C\eps\omega_0]\in Int(\Gamma)$ for small enough $\eps$.  Then the argument is the same as (1),  by sending $\eps\rightarrow 0$ and using the continuity of $f$.
\end{proof}

In general,  a viscosity solution is only continuous,  and there is no guarantee that it is punctually second order differentiable anywhere.  So our first step would be to find suitable regularization,  that makes the original solution a subsolution/supersolution,  and that it is punctually second order differentiable a.e.  

Henceforth we assume that $\varphi\in C(M)$ is a viscosity solution to $F(\chi+dd^c\varphi)=e^{G(x,\varphi)}$.  We define the super convolution of $\varphi$ as follows:
\begin{equation}\label{4.1}
\varphi^{\eps}(z)=\sup_{\xi\in T_zM}\big(\varphi(\exp_z(\xi))+\eps-\frac{1}{\eps}|\xi|_z^2\big),
\end{equation}
where $\exp_z(\xi)$ is the exponential map at $z$,  defined using the metric $\omega_0$,  and $|\xi|_z$ denotes the length of the tangent vector $\xi$,  again using the metric $\omega_0$.  Similarly we define the inf convolution of $\varphi$:
\begin{equation}\label{4.2}
\varphi_{\eps}(z)=\inf_{\xi\in T_zM}\big(\varphi(\exp_z(\xi))-\eps+\frac{1}{\eps}|\xi|_z^2\big).
\end{equation}

A key step is to show that $\varphi^{\eps}$ and $\varphi_{\eps}$ defined above produce viscosity subsolution and supersolution,  up to a small error.  Also from the definition of $\varphi^{\eps}$ and $\varphi_{\eps}$,  they will be semi-convex and semi-concave,  hence punctually second order differentiable a.e.  Next we make this precise.
We first verify that $\varphi^{\eps}$ and $\varphi_{\eps}$ defined above are subsolution/supersolution.

 Denote $\rho_{\varphi}(r),\,\,0<r<1$ to be the modulus of continuity of $\varphi$.  That is,  $\rho_{\varphi}(r)=\max\{|\varphi(x)-\varphi(y)|:\,\,d_g(x,y)\le r,\,x,y\in M\}$.
\begin{prop}\label{p4.3}
Let $\varphi$ be a viscosity solution to $F(\chi+dd^c \varphi)=e^{G(x,\varphi)}$ with $G(x,\varphi)$ continuous.  Then there exist continuous functions $\rho(\eps):(0,1)\rightarrow \bR_{>0}$ with $\rho(0+)=0$,  and $\rho_1(\eps,a_1),\,\rho_2(\eps,a_2):(0,1)^2\rightarrow \bR_{>0}$ with $\rho_i(0+,0+)=0,\,\,i=1,2$,  such that for any $0<\eps<1$,  $0<a_i<1,\,i=1,\,2$:
\begin{enumerate}
\item $\frac{\varphi^{\eps}}{1+a_1}$ is $\Gamma$-subharmonic with respect to $\frac{\chi+\rho(\eps)\omega_0}{1+a_1}$ in the viscosity sense,  \sloppy and $F(\frac{\chi+\rho(\eps)\omega_0}{1+a_1}+dd^c\frac{\varphi^{\eps}}{1+a_1})\ge e^{G(x,\frac{\varphi^{\eps}}{1+a_1})}-\rho_1(\eps,a_1)$ in the viscosity sense.
\item $\frac{\varphi_{\eps}}{1-a_2}$ satisfies $F(\frac{\chi-\rho(\eps)\omega_0}{1-a_2}+dd^c\frac{\varphi_{\eps}}{1-a_2})\le e^{G(x,\frac{\varphi_{\eps}}{1-a_2})}+\rho_2(\eps,a_2)$ in the viscosity sense.
\end{enumerate}
Here the functions $\rho,\,\rho_i,\,i=1,\,2$ are determined by $\rho_{\varphi}$,  $||\varphi||_{L^{\infty}}$,  the form $\chi$ and the background metric.
\end{prop}
For the proof of Proposition \ref{p4.3},  we need to understand how the touching of $\varphi^{\eps}$ or $\varphi_{\eps}$ from above or below translates to the touching of $\varphi$.

\begin{lem}\label{l4.4}
\begin{enumerate}
\item Let $\varphi\in C(M)$ and we define $\varphi^{\eps}$ according to (\ref{4.1}).  Let $x_0\in M$ and $P$ is a $C^2$ function defined in a neighborhood $x_0$ that touches $\varphi^{\eps}$ from above at $x_0$.  Assume that
\begin{equation*}
\varphi^{\eps}(x_0)=\varphi\big(\exp_{x_0}(\xi_0)\big)+\eps-\frac{1}{\eps}|\xi|_{x_0}^2,\text{ for some $\xi_0\in T_{x_0}M$}.
\end{equation*}
Let $\xi(z)\in T_zM$ be a smooth vector field defined in a neighborhood of $x_0$ with $\xi(x_0)=\xi_0$.  Put $\phi(z)=\exp_z(\xi(z))$,  then $z\mapsto P-\eps+\frac{1}{\eps}|\xi(z)|_z^2$ touches $\varphi\circ \phi$ from above at $x_0$.
\item Let $\varphi\in C(M)$ and we define $\varphi_{\eps}$ according to (\ref{4.2}).  Let $x_0\in M$ and $P
$ is a $C^2$ function defined in a neighgorhood $x_0$ that touches $\varphi_{\eps}$ from below at $x_0$.  Assume that:
\begin{equation*}
\varphi_{\eps}(x_0)=\varphi\big(\exp_{x_0}(\tilde{\xi}_0)\big)-\eps+\frac{1}{\eps}|\tilde{\xi}|_{x_0}^2,\text{ for some $\tilde{\xi}_0\in T_{x_0}M$}.
\end{equation*}
Let $\tilde{\xi}(z)\in T_zM$ be a smooth vector field defined in a neighborhood of $x_0$ with $\tilde{\xi}(x_0)=\tilde{\xi}_0$.  Put $\tilde{\phi}(z)=\exp_z(\tilde{\xi}(z))$,  then $z\mapsto P+\eps-\frac{1}{\eps}|\tilde{\xi}(z)|^2_z$ touches $\varphi\circ \phi$ from below at $x_0$.
\end{enumerate}
\end{lem}
\begin{proof}
We will just prove (1),  and the proof of (2) repeats that of (1) almost word for word.
By assumption,  we know that $P(z)\ge \varphi^{\eps}(z)$ in a neighorhood of $x_0$.  Therefore
\begin{equation*}
\begin{split}
P(z)&\ge \sup_{\xi\in T_zM}\big(\varphi(\exp_z(\xi))+\eps-\frac{1}{\eps}|\xi|_z^2\big)\ge \varphi(\exp_z(\xi(z))+\eps-\frac{1}{\eps}|\xi(z)|_z^2\\
&=\varphi\circ \phi(z)+\eps-\frac{1}{\eps}|\xi(z)|_z^2.
\end{split}
\end{equation*}
Also we know that equality is achieved when $z=x_0$.  This follows from that $\xi(x_0)=\xi_0$,  and $\xi_0$ achieves the sup by assumption.
\end{proof}
Next we wish to choose an appropriate $\phi(z)$,  and obtain $\xi(z)$ by inverting $\exp_z$.  Denote $w_0=\exp_z(\xi_0)$,  we hope to define $\phi(z)$ in a neighhorhood of $x_0$ such that:
\begin{equation*}
\sum_{i,j}g_{i\bar{j}}(w_0)\frac{\partial \phi_i}{\partial z_a}\frac{\partial \bar{\phi}_j}{\partial\bar{z}_b}=g_{a\bar{b}}(x_0),\,\,\text{$\phi$ is holomorphic},\,\,\det D_z\phi(x_0)\neq 0.
\end{equation*}
Next we explain why this choice is possible.  We first wish to estimate the smallness of $\xi$.  For this we have
\begin{lem}\label{l4.5}
\begin{enumerate}
\item Let $\varphi\in C(M)$ and let $\varphi^{\eps}$ be as defined by (\ref{4.1}).  Let $x_0\in M$ and $\xi_0\in T_{x_0}M$.  Assume that $\xi_0$ achieves the sup in the definition of $\varphi^{\eps}$ at $x_0$,  then 
\begin{equation*}
|\xi_0|_{x_0}\le \eps^{\frac{1}{2}}\big(\rho_{\varphi}(C\eps^{\frac{1}{2}})\big)^{\frac{1}{2}}.
\end{equation*}
\item Let $\varphi_{\eps}$ be as defined by (\ref{4.2}) and $\tilde{\xi}_0\in T_{x_0}M$ achieves the inf in the definition of $\varphi_{\eps}$ at $x_0$,  then 
\begin{equation*}
|\tilde{\xi}_0|_{x_0}\le \eps^{\frac{1}{2}}\big(\rho_{\varphi}(C\eps^{\frac{1}{2}})\big)^{\frac{1}{2}}.
\end{equation*}
Here the constant $C$ appearing above depends only on $||\varphi||_{L^{\infty}}$ and the background metric.
\end{enumerate}
\end{lem}
\begin{proof}
Again we will just prove (1),  and the proof of (2) repeats that of (1) word for word.  Indeed,  we have:
\begin{equation}\label{4.3}
\varphi^{\eps}(x_0)=\varphi(\exp_{x_0}(\xi_0))+\eps-\frac{1}{\eps}|\xi_0|_{x_0}^2\ge \varphi(x_0)+\eps.
\end{equation}
In the $\ge$ above,  we have taken $\xi=0$.  Therefore,  we get:
\begin{equation*}
|\xi_0|_{x_0}\le \eps^{\frac{1}{2}}(2||\varphi||_{L^{\infty}})^{\frac{1}{2}}.
\end{equation*}
Denote $C=(2||\varphi||_{L^{\infty}})^{\frac{1}{2}}$,  we go back to (\ref{4.3}) and obtain that 
\begin{equation*}
|\xi_0|_{x_0}^2\le \eps\big(\varphi(\exp_{x_0}(\xi_0))-\varphi(x_0)\big)\le \eps\rho_{\varphi}\big(d_g(\exp_{x_0}(\xi_0),x_0)\big)\le \eps\rho_{\varphi}(|\xi_0|_{x_0})\le \eps\rho_{\varphi}(C\eps^{\frac{1}{2}}).
\end{equation*}
Hence our result follows.
\end{proof}
From Lemma \ref{l4.5},  we see that,  as $\eps\rightarrow 0$,  we would have $\xi_0,\,\,\tilde{\xi}_0\rightarrow x_0$.  

Another observation that follows from Lemma \ref{l4.5} is that both $\varphi^{\eps}$ and $\varphi_{\eps}$ approximate $\varphi$ uniformly.
\begin{lem}
$\varphi^{\eps}$ and $\varphi_{\eps}$ approximate $\varphi$ uniformly as $\eps\rightarrow0$.  More precisely,  for any $x_0\in M$:
\begin{equation*}
\begin{split}
&\varphi(x_0)+\eps\le \varphi^{\eps}(x_0)\le \varphi(x_0)+\eps+\rho_{\varphi}(C\eps^{\frac{1}{2}}),\\
&\varphi(x_0)-\eps-\rho_{\varphi}(C\eps^{\frac{1}{2}})\le \varphi_{\eps}(x_0)\le \varphi(x_0)-\eps.
\end{split}
\end{equation*}
Here $C$ is the same constant as in Lemma \ref{l4.5}.
\end{lem}
\begin{proof}
By taking $\xi=0$,  we see that $\varphi^{\eps}(x_0)\ge \varphi(x_0)+\eps$,  $\varphi_{\eps}(x_0)\le \varphi(x_0)-\eps$.

For the other inequality,  we note that:
\begin{equation*}
\begin{split}
&\varphi^{\eps}(x_0)=\varphi(\exp_{x_0}(\xi_0))+\eps-\frac{1}{\eps}|\xi_0|_{x_0}^2\le \varphi(\exp_{x_0}(\xi_0))+\eps\\
&\le \varphi(x_0)+\rho_{\varphi}\big(d_g(x_0,\exp_{x_0}(\xi_0))\big)+\eps\le \varphi(x_0)+\rho_{\varphi}(C\eps^{\frac{1}{2}})+\eps.
\end{split}
\end{equation*}

\end{proof}
Now we choose a coordinate chart near $x_0$ such that $x_0$ is represented by the origin,  and that:
\begin{equation*}
g_{a\bar{b}}(x_0)=\delta_{ab}.
\end{equation*}
Let $N$ be an invertible $n\times n$ matrix,  such that 
\begin{equation}\label{4.4}
\sum_{i,j}N_{ia}\bar{N}_{jb}g_{i\bar{j}}(w_0)=\delta_{ab}=g_{a\bar{b}}(x_0),\,\,\,w_0=\exp_{x_0}(\xi_0).
\end{equation}
Moreover,  we may assume that:
\begin{equation}\label{4.5N}
|N-I|\le C_n|g_{i\bar{j}}(w_0)-\delta_{ij}|\le C_nCd_g(w_0,x_0)\le C_nC|\xi_0|_{x_0}.
\end{equation}
In the above,  the $C$ depends only on the background metric.  Now we wish to take:
\begin{equation}\label{def phi}
\phi(z)=\exp_{x_0}(\xi_0)+N\cdot z.
\end{equation}
Next,  we need to show the existence of a local vector field such that $\exp_z(\xi(z))=\phi(z)$.  
\begin{lem}\label{l4.6}
\begin{enumerate}
\item There exists a smooth vector field $\xi(z)\in T_zM$,  defined in a neighborhood of $x_0$,  such that 
\begin{equation*}
\exp_z(\xi(z))=\phi(z).
\end{equation*}
\item There exist a constant $C>0$,  depending only on the background metric,  such that:
\begin{equation*}
|D_z\xi|(x_0)\le C|\xi_0|_{x_0},\,\,|D_z^2\xi|(x_0)\le C|\xi_0|_{x_0},\,\,|D_z^2(|\xi(z)|_z^2)|\le C|\xi_0|_{x_0}^2.
\end{equation*}
\end{enumerate}
\end{lem}
We postpone the proof of this lemma for the moment and explain first how to use this lemma to finish the proof of Proposition \ref{p4.3}.
We still need one more lemma,  which justifies our choice of the map $\phi$.
\begin{lem}\label{l4.7}
Let $\phi$ be defined by (\ref{def phi}) in a neighborhood of $x_0$ with $N$ given by (\ref{4.4}).  Denote $w_0=\exp_{x_0}(\xi_0)$,  then for any function $Q$ defined in a neighborhood of $x_0$,  one has:
\begin{equation*}
\sum_kg^{i\bar{k}}(Q\circ \phi^{-1}\big)_{j\bar{k}}(w_0)=\sum_{a,b,p,q}N_{iq}g^{q\bar{b}}Q_{a\bar{b}}(x_0)(N^{-1})_{aj}.
\end{equation*}
\end{lem}
\begin{proof}
This is a straightforward calculation.  Indeed,

\begin{equation}\label{4.8}
\begin{split}
&\sum_kg^{i\bar{k}}(w_0)\big(Q\circ \phi^{-1}\big)_{j\bar{k}}(w_0)=\sum_{a,b}\sum_kg^{i\bar{k}}(w_0)Q_{a\bar{b}}(x_0)\frac{\partial(\phi^{-1})_a}{\partial z_j}\frac{\partial(\phi^{-1})_{\bar{b}}}{\partial\bar{z}_k}\\
&=\sum_{a,b}\sum_kg^{i\bar{k}}(w_0)(N^{-1})_{aj}\overline{(N^{-1})_{bk}}Q_{a\bar{b}}(x_0)\\
&=\sum_{a,b}\sum_k\sum_{p,q}\bar{N}_{kp}g^{q\bar{p}}(x_0)N_{iq}(N^{-1})_{aj}\overline{(N^{-1})_{bk}}Q_{a\bar{b}}(x_0)\\
&=\sum_{a,b,p,q}N_{iq}g^{q\bar{b}}(x_0)Q_{a\bar{b}}(x_0)(N^{-1})_{aj}.
\end{split}
\end{equation}
Some explanations are in order.  In the first equality,  we noted that $\phi(z)$ as given by (\ref{def phi}) is holomorphic.  In the third equality,  we used (\ref{4.4}).  In the last equality,  we noted that $\sum_k\bar{N}_{kp}\overline{(N^{-1})_{bk}}=\delta_{pb}$.
\end{proof}
Another thing we observe is that:
\begin{lem}\label{l4.8}
Denote $w_0=\exp_{x_0}(\xi_0)$,  then for any $i,\,j$,  we have
\begin{equation*}
|g^{i\bar{k}}\chi_{j\bar{k}}(w_0)-N_{iq}g^{q\bar{b}}\chi_{a\bar{b}}(x_0)(N^{-1})_{aj}|\le C|\xi_0|_{x_0}.
\end{equation*}
In the above,  $C$ depends only on the background metric and the form $\chi$.
\end{lem}
\begin{proof}
To see this,  we simply write this as a telescoping sum:
\begin{equation*}
\big(g^{i\bar{k}}\chi_{j\bar{k}}(w_0)-g^{i\bar{k}}\chi_{j\bar{k}}(x_0)\big)+\big(\delta_{iq}g^{q\bar{b}}\chi_{a\bar{b}}(x_0)\delta_{aj}-N_{iq}g^{q\bar{b}}\chi_{a\bar{b}}(x_0)(N^{-1})_{aj}\big).
\end{equation*}
The first bracket above is clearly bounded by $Cd_g(w_0,x_0)\le C'|\xi_0|_{x_0}$.  For the second bracket,  we need to use (\ref{4.5N}) to see that it is also bounded by $C|\xi_0|_{x_0}$.
\end{proof}

\begin{proof}
(Of Proposition \ref{p4.3}) First we prove (1).  Let $P$ be a $C^2$ function that touches $\frac{\varphi^{\eps}}{1+a_1}$ from above at $x_0$,  which is the same as saying $(1+a_1)P$ touches $\varphi^{\eps}$ from above at $x_0$.  We need to show that:
\begin{equation}\label{4.6}
\begin{split}
&\lambda\big(\frac{\chi+\rho(\eps)\omega_0}{1+a_1}+dd^c P\big)\in \Gamma,\\
&f\big(\lambda[\frac{\chi+\rho(\eps)\omega_0}{1+a_1}+dd^cP](x_0)\big)\ge e^{G(x_0,P(x_0))}-\rho_1(\eps,a_1).
\end{split}
\end{equation}
The choice of $\rho(\eps),\,\rho_1(\eps,a_1)$ will be made clear later on.  First we can see from Lemma \ref{l4.4},  part (1) that the function $(1+a_1)P-\eps+\frac{1}{\eps}|\xi(z)|_z^2$ touches $\varphi\circ \phi$ from above.  Note that $\phi(x_0)=\exp_{x_0}(\xi_0)$ and that $\det D_z\phi(x_0)\neq 0$,  we see that $\phi$ defines an invertible map between a neighborhood of $x_0$ and a neighborhood of $\exp_{x_0}(\xi_0)$.  

Therefore,  we see that $\big((1+a_1)P-\eps+\frac{1}{\eps}|\xi(z)|^2\big)\circ \phi^{-1}(w)$ touches $\varphi$ from above at $\phi(x_0)=\exp_{x_0}(\xi_0)$.
Therefore,  using that $\varphi$ is a viscosity solution,  we see that: with $w_0=\exp_{x_0}(\xi_0)$,
\begin{equation}\label{4.7}
\begin{split}
&\lambda[\chi+dd^c((1+a_1)P-\eps+\frac{1}{\eps}|\xi(z)|^2_z)\circ \phi^{-1}](w_0)\in \Gamma,\\
&f\big(\lambda[\chi+dd^c((1+a_1)P-\eps+\frac{1}{\eps}|\xi(z)|^2_z)\circ \phi^{-1}](w_0)\big)\ge e^{G(w_0,((1+a_1)P-\eps+\frac{1}{\eps}|\xi(z)|^2_z)\circ \phi^{-1}(w_0))}.
\end{split}
\end{equation}
We wish to show that (\ref{4.7}) implies (\ref{4.6}).
For this,  we may calculate,  for fixed $i,\,j$:
\begin{equation}\label{4.10N}
\begin{split}
&g^{i\bar{k}}\big(\chi_{j\bar{k}}+\big(((1+a_1)P-\eps+\frac{1}{\eps}|\xi|^2_z)\circ \phi^{-1}\big)_{j\bar{k}}\big)(w_0)\\
&=g^{i\bar{k}}\chi_{j\bar{k}}(w_0)+N_{iq}g^{q\bar{b}}((1+a_1)P-\eps+\frac{1}{\eps}|\xi(z)|^2_z)_{a\bar{b}}(x_0)(N^{-1})_{aj}\\
&\le N_{iq}g^{q\bar{b}}(\chi_{a\bar{b}}(x_0)+C|\xi_0|_{x_0}g_{a\bar{b}}(x_0))(N^{-1})_{aj}\\
&+N_{iq}g^{q\bar{b}}((1+a_1)P-\eps+\frac{1}{\eps}|\xi(z)|^2_z)_{a\bar{b}}(x_0)(N^{-1})_{aj}
\end{split}
\end{equation}
In the equality above,  we used Lemma \ref{l4.7} with $Q=(1+a_1)P-\eps+\frac{1}{\eps}|\xi(z)|_z^2$.
The inequality above follows from Lemma \ref{l4.8}.
Then we see from (\ref{4.10N}) that:
\begin{equation}\label{4.11}
\lambda[\chi+dd^c((1+a_1)P-\eps+\frac{1}{\eps}|\xi(z)|^2_z)\circ \phi^{-1}]\le \lambda\big(\chi+C|\xi_0|_{x_0}\omega_0+dd^c((1+a_1)P+\frac{1}{\eps}|\xi(z)|_z^2)\big).
\end{equation}
The meaning of the above inequality is that the difference belongs to $\Gamma_n$.  Moreover,  we may use Lemma \ref{l4.5} and \ref{l4.6} to see that:
\begin{equation}\label{4.12}
\begin{split}
&\chi+C|\xi_0|_{x_0}\omega_0+dd^c((1+a_1)P+\frac{1}{\eps}|\xi(z)|^2_z)\\
&\le \chi+C'(\rho_{\varphi}^{\frac{1}{2}}(C''\eps^{\frac{1}{2}})+\rho_{\varphi}(C''\eps^{\frac{1}{2}}))\omega_0+(1+a_1)dd^c P.
\end{split}
\end{equation}
Therefore,  if we now put $\rho(\eps)=C'(\rho_{\varphi}^{\frac{1}{2}}(C''\eps^{\frac{1}{2}})+\rho_{\varphi}(C''\eps^{\frac{1}{2}}))$,  we see that:
\begin{equation*}
\lambda\big(\chi+\rho(\eps)\omega_0+(1+a_1)dd^c P\big)\in \Gamma.
\end{equation*}
This is exactly the first statement of (\ref{4.6}).  For the second statement,  the above calculation already implies that:
\begin{equation}\label{4.13}
\begin{split}
&f\big(\lambda[\chi+dd^c((1+a_1)P-\eps+\frac{1}{\eps}|\xi(z)|^2_z)\circ \phi^{-1}](w_0)\big)\le f\big(\lambda\big(\chi+\rho(\eps)\omega_0+(1+a_1)dd^c P\big)\big)\\
&=(1+a_1)f\big(\lambda[\frac{\chi+\rho(\eps)\omega_0}{1+a_1}+dd^cP]\big).
\end{split}
\end{equation}
The equality here uses that $f$ is of homogeneity one.
On the other hand,
\begin{equation}\label{4.14}
\begin{split}
&\frac{1}{1+a_1}e^{G(w_0,((1+a_1)P-\eps+\frac{1}{\eps}|\xi(z)|^2_z)\circ \phi^{-1}(w_0))}=\frac{1}{1+a_1}e^{G(w_0,((1+a_1)P(x_0)-\eps+\frac{1}{\eps}|\xi_0|^2_{x_0}))}\\
&\ge e^{G(x_0,P(x_0))}-\rho_1(\eps,a_1).
\end{split}
\end{equation}
Here we again used Lemma \ref{l4.5} on the estimate of $|\xi_0|_{x_0}$.
Combining (\ref{4.13}) and (\ref{4.14}) gives the second statement of (\ref{4.6}).  So far we have proved the first part of Proposition \ref{p4.3}.
The second part of Proposition \ref{p4.3} is proved similarly,  which we sketch briefly.

Let $P$ be a $C^2$ function defined in a neighborhood of $x_0$,  which touches $\frac{\varphi_{\eps}}{1-a_2}$ from below at $x_0$.  We need to show that:
Either
\begin{equation}\label{5.16}
\lambda\big(\frac{\chi-\rho(\eps)\omega_0}{1-a_2}+dd^c P\big)(x_0)\notin \Gamma
\end{equation}
or
\begin{equation}\label{4.15}
\begin{split}
&f\big(\lambda[\frac{\chi-\rho(\eps)\omega_0}{1-a_2}+dd^c P](x_0)\big)\le e^{G(x_0,P(x_0)}+\rho_2(\eps,a_2),\\
& \lambda\big(\frac{\chi-\rho(\eps)\omega_0}{1-a_2}+dd^c P\big)(x_0)\in \Gamma
\end{split}
\end{equation}
Since $(1-a_2)P$ touches $\varphi_{\eps}$ from below at $x_0$,  we see from Lemma \ref{l4.4},  part (2) that $\big((1-a_2)P+\eps-\frac{1}{\eps}|\tilde{\xi}(z)|^2_z\big)\circ \tilde{\phi}^{-1}(w)$ touches $\varphi$ from below at $\tilde{\phi}(x_0)=\exp_{x_0}(\tilde{\xi}_0)$ (which we denote as $\tilde{w}_0$ from now on).  Here $\tilde{\xi}_0$,  $\tilde{\xi}(z)$,  $\tilde{\phi}(z)$ is defined in the same way as $\varphi^{\eps}$,  hence satisfy the same estimates as $\xi_0,\,\xi(z),\,\phi(z)$.  Therefore,  we may conclude,  as before:
Either
\begin{equation}\label{4.16}
\lambda[\chi+dd^c((1-a_2)P+\eps-\frac{1}{\eps}|\tilde{\xi}(z)|^2_z)\circ\tilde{\phi}^{-1}](\tilde{w}_0)\notin \Gamma
 \end{equation}
 or
 \begin{equation}\label{5.19}
 \begin{split}
 &f\big(\lambda[\chi+dd^c((1-a_2)P+\eps-\frac{1}{\eps}|\tilde{\xi}(z)|_z^2)\circ \tilde{\phi}^{-1}](\tilde{w}_0)\le e^{G(\tilde{w}_0,((1-a_2)P+\eps-\frac{1}{\eps}|\tilde{\xi}(z)|^2)\circ \tilde{\phi}^{-1}(\tilde{w}_0))},\\
& \lambda[\chi+dd^c((1-a_2)P+\eps-\frac{1}{\eps}|\tilde{\xi}(z)|^2_z)\circ\tilde{\phi}^{-1}](\tilde{w}_0)\in \Gamma.
\end{split}
\end{equation}

So we just need to deduce (\ref{5.16}) or (\ref{4.15}) from (\ref{4.16}) or (\ref{5.19}).  Similar calculations as in the proof of part (1) will show that:
\begin{equation}
\lambda\big(\chi-\rho(\eps)\omega_0+dd^c(1-a_2)P\big)(x_0)\le \lambda[\chi+dd^c((1-a_2)P+\eps-\frac{1}{\eps}|\tilde{\xi}(z)|^2_z)\circ\tilde{\phi}^{-1}](\tilde{w}_0).
\end{equation}
Here one can actually make $\rho(\eps)$ to be the same as part (1).  Therefore,  if $\lambda[\chi-\rho(\eps)\omega_0+dd^c(1-a_2)P]\in \Gamma$,  it will imply $\lambda[\chi+dd^c((1-a_2)P+\eps+\frac{1}{\eps}|\tilde{\xi}(z)|^2_z)\circ\tilde{\phi}^{-1}](\tilde{w}_0)\in \Gamma$.  Moreover,  in this case,  we also have:
\begin{equation*}
\begin{split}
&f\big(\lambda\big(\frac{\chi-\rho(\eps)\omega_0}{1-a_2}+dd^c P\big)\big)\le \frac{1}{1-a_2}f\big(\lambda[\chi+((1-a_2)P+\eps-\frac{1}{\eps}|\tilde{\xi}(z)|^2_z)\circ \tilde{\phi}^{-1}]\big)(\tilde{w}_0)\\
&\le \frac{1}{1-a_2}e^{G(\tilde{w}_0,((1-a_2)P+\eps-\frac{1}{\eps}|\tilde{\xi}(z)|^2)\circ \tilde{\phi}^{-1}(\tilde{w}_0))}
\end{split}
\end{equation*}
The same calculation as in the proof of part (1) shows that one can estimate the right hand side from above by $e^{G(x_0,P(x_0))}+\rho_2(\eps,a_2)$ (again we note that $\tilde{\phi}^{-1}(\tilde{w}_0)=x_0$).
\end{proof}

Now let us prove Lemma \ref{l4.6}:
\begin{proof}
(Of Lemma \ref{l4.6})
The existence part of the vector field $\xi(z)$ is a result of implicit function theorem.  Let $U$ be an open subset of $TM$,   such that $(x_0,\xi_0)\in U$ and that $\exp_z(\xi),\,\phi(z)$ is inside the coordinate chart near $x_0$.  We consider the following map:
\begin{equation*}
\mathcal{F}:U\rightarrow \bC^n,\,\,\,(z,\xi)\mapsto \exp_z(\xi)-\phi(z).
\end{equation*}
On the right hand side above,  we have identified $\exp_z(\xi)$ and $\phi(z)$ with points in $\bC^n$ using the coordinate chart,  so that the subtraction makes sense.  Moreover,
\begin{equation*}
D_{\xi}\mathcal{F}|_{(z,\xi)(q)=(x_0,\xi_0)}=D_{\xi}(\exp_z\xi)|_{z=x_0,\,\xi=\xi_0}(q),\,\,\,q\in T_{x_0}M.
\end{equation*}
From the lemma \ref{l5.13New} below,  we know that:
\begin{equation}\label{5.19New}
D_{\xi}(\exp_z(\xi))|_{z=x_0,\,\xi=0}=I,\,D_z(\exp_z(\xi))|_{z=x_0,\,\xi=0}=I,\,D_{zz}(\exp_z(\xi))|_{z=x_0,\,\xi=0}=0.
\end{equation}
Therefore $D_{\xi}(\exp_z\xi)_{z=x_0,\,\xi=\xi_0}$ would be non-singular,  since $\xi_0$ is very close to 0 due to Lemma \ref{l4.5}.  Moreover,  we also know that $\mathcal{F}(x_0,\xi_0)=0$.  Therefore,  we may conclude from implicit function theorem that there is a neighborhood $V$ of $x_0$,  and a vector field $\xi(z),\,z\in V$,  such that $\mathcal{F}(z,\xi(z))=0$.  Namely $\phi(z)=\exp_z(\xi(z))$.

Now we derive the estimates of $\xi(z)$.  By differentiation,  we see that:
\begin{equation*}
N=D_z\phi(z)|_{z=x_0}=D_z(\exp_z\xi)|_{z=x_0,\xi=\xi_0}+D_{\xi}(\exp_z\xi)D_z\xi|_{z=x_0,\xi=\xi_0}
\end{equation*}
From (\ref{5.19New}) and the smoothness of the exponential map,  we know that $|D_z(\exp_z\xi)-I|_{z=x_0,\,\xi=\xi_0}\le C|\xi_0|_{x_0}$.  Hence we may use (\ref{4.5N}) to see that:
\begin{equation*}
|N-D_z(\exp_z(\xi))_{z=x_0,\,\xi=\xi_0}|\le C|\xi_0|_{x_0}.
\end{equation*}
Here the $C$ depends only on the background metric.  Also we noted that $D_{\xi}(\exp_z\xi)|_{z=x_0,\,\xi=\xi_0}$ is invertible,  we see that $|D_z\xi|_{z=x_0}\le C|\xi_0|_{x_0}$.  Differentiating once more,  we get:
\begin{equation*}
\begin{split}
&0=D_{zz}(\exp_z\xi)|_{z=x_0,\,\xi=\xi_0}+2D_{z\xi}(\exp_z\xi)|_{z=x_0,\xi=\xi_0}D_z\xi|_{z=x_0}\\
&+D_{\xi\xi}(\exp_z\xi)|_{z=x_0,\,\xi=\xi_0}D_z\xi*D_z\xi|_{z=x_0}+D_{\xi}(\exp_z\xi)|_{z=x_0,\xi=\xi_0}D_z^2\xi|_{z=x_0}.
\end{split}
\end{equation*}
In the above,  $D_z\xi*D_z\xi$ just denotes some quadratic expression of $D_z\xi$.  Using again (\ref{5.19New}) and the smoothness of the exponential map,  we know that:
\begin{equation*}
|D_{zz}(\exp_z\xi)|_{z=x_0,\,\xi=\xi_0}\le C|\xi_0|_{x_0},\,\,|D_{z\xi}(\exp_z\xi)|_{z=x_0,\,\xi=\xi_0}\le C,\,|D_{\xi\xi}(\exp_z\xi)|_{z=x_0,\,\xi=\xi_0}\le C.
\end{equation*}
From this,  we see that
\begin{equation*}
|D_z^2\xi|_{z=x_0}\le C|\xi_0|_{x_0}.
\end{equation*}
Finally,  
\begin{equation*}
\begin{split}
&D_z^2\big(|\xi(z)|_z^2\big)=D_z^2\big(g_{i\bar{j}}(z)\xi_i(z)\bar{\xi}_j(z)\big)=D_z^2g_{i\bar{j}}(z)\xi_i(z)\bar{\xi}_j(z)+D_zg_{i\bar{j}}D_z\xi_i\bar{\xi}_j\\
&+D_zg_{i\bar{j}}\xi_iD_z\bar{\xi}_j+g_{i\bar{j}}\big(D_z^2\xi_i\bar{\xi}_j+\xi_iD_z^2\bar{\xi}_j+D_z\xi_i*D_z\bar{\xi}_j\big).
\end{split}
\end{equation*}
Using the above estimates for $D_z\xi$ and $D_z^2\xi$,  we see that $|D_z^2(|\xi(z)|_z^2)|_{z=x_0}\le C|\xi_0|_{x_0}^2$.
\end{proof}
In the above,  we used the following lemma about the Taylor expansion of the exponential map.  This lemma can be found in \cite{D}.
\begin{lem}\label{l5.13New}
The exponential map on a Hermitian manifold has the Taylor expansion in the following form under local coordinates:
\begin{equation*}
\exp_z(\xi)_m=g_m(z,\zeta)+\sum_{j,k,l}c_{jklm}(\frac{1}{2}\bar{z}_k+\frac{1}{6}\bar{\zeta}_k)\zeta_j\zeta_l+O\big(|\zeta|^2(|z|+|\zeta|)^2),
\end{equation*}
where
\begin{equation*}
g_m(z,\zeta)=z_m+\zeta_m-\sum_{j,l}a_{jlm}z_j\zeta_l+\sum_{j,k,l,p}a_{jlp}a_{kpm}z_jz_k\zeta_l-\sum_{j,k,l}b_{jklm}(z_jz_k\zeta_l+z_k\zeta_j\zeta_l+\frac{1}{3}\zeta_j\zeta_k\zeta_l),
\end{equation*}
and $\xi$ and $\zeta$ are related through:
\begin{equation*}
\zeta_m=\xi_m+\sum_{j,l}a_{jlm}z_j\xi_l+\sum_{j,k,l}b_{jklm}z_jz_k\xi_l.
\end{equation*}
In the above,  $(\exp_z\xi)_m$ denotes the $m$-th component of the exponential map under local coordinates. 
\end{lem}
Finally we observe that $\varphi^{\eps}$ and $\varphi_{\eps}$ are punctually second order differentiable a.e.  Indeed,  one has:
\begin{lem}\label{l4.11N}
Let $\varphi\in C(M)$,  and we define $\varphi^{\eps},\,\varphi_{\eps}$ according to (\ref{4.1}) and (\ref{4.2}).  Let $x_0\in M$ and we choose local coordinates in a neighborhood of $x_0$.  Then there exists a neighborhood $U$ of $x_0$,  and $C_{\eps}>0$,  such that: $z\mapsto \varphi^{\eps}(z)+C_{\eps}|z|^2$ is convex on $U$ under the coordinates,  and $z\mapsto \varphi_{\eps}(z)-C_{\eps}|z|^2$ is concave on $U$ under the coordinates.  In particular,  $\varphi^{\eps}$,  $\varphi_{\eps}$ are punctually second order differentiable a.e.  
\end{lem}
\begin{proof}
We just prove that $\varphi^{\eps}$ is semi-convex.  The proof that $\varphi_{\eps}$ is semi-concave follows similar lines.  We can choose $U_0$ small enough,  such that for any $z,\,w\in U_0$,  there is a unique $\xi\in T_zM$ such that $\exp_z(\xi)=w$.  Moreover,  we can assume that $\xi$ depends smoothly on $z$ and $w$,  and that $|\xi|^2_z$ is also smooth in $z$ and $w$.  Therefore,  one has,  for some neighborhood $U$ of $x_0$ (possibly smaller than $U_0$):
\begin{equation}\label{4.18}
\begin{split}
&\varphi^{\eps}(z)=\sup_{\xi\in T_zM}\big(\varphi(\exp_z(\xi)+\eps-\frac{1}{\eps}|\xi|^2_z)\big)=\sup_{\xi\in T_zM,\,|\xi|<r_0}\big(\varphi(\exp_z(\xi))+\eps-\frac{1}{\eps}|\xi|^2\big)\\
&=\sup_{w\in U_0}\big(\varphi(w)+\eps-\frac{1}{\eps}|\xi(z,w)|^2_z\big),\,\,z\in U.
\end{split}
\end{equation}
The second equality used Lemma \ref{l4.5} on the estimate of $\xi$ that achieves the sup.  In the last inequality,  we noted that for some neighborhood $U$ of $x_0$,  the image of $U_0\ni w\mapsto \xi(z,w)$ will cover $\{\xi\in T_zM:|\xi|_z<r_0\}$ for any $z\in U$.  
Note that in (\ref{4.18}),  the function $z\mapsto \varphi(w)+\eps-\frac{1}{\eps}|\xi(z,w)|^2_z$ is smooth,  and has uniform $C^2$ bound (in $z$,  uniform with respect to $w$).  Therefore,  taking sup will imply that $\varphi^{\eps}$ is semi-convex.
\end{proof}

\subsection{when the right hand side has strict monotonicity}
In this subsection,  we assume that $G(x,\varphi)$ is continuous and strictly monotone increasing in $\varphi$.  We wish to show that:
\begin{thm}\label{t4.1}
Let $G(x,\varphi)$ be continuous,  and strictly monotone increasing in $\varphi$.  Then there exists at most one viscosity solution to $F(\chi+dd^c\varphi)=e^{G(x,\varphi)}$.
\end{thm}
If not,  then there exist two viscosity solutions $\varphi_1,\,\varphi_2$,  and $\varphi_1\neq \varphi_2$.  Without loss of generality,  we may assume that:
\begin{equation}
\kappa_0:=\max_M(\varphi_1-\varphi_2)>0.
\end{equation}
Now we consider the super-convolution,  applied to $\varphi_1$,  and the inf-convolution,  applied to $\varphi_2$.  We define:
\begin{equation*}
\kappa_{\eps,a_1,a_2}=\max_M\big(\frac{(\varphi_1)^{\eps}}{1+a_1}-\frac{(\varphi_2)_{\eps}}{1-a_2}\big).
\end{equation*}
Then it is easy to see that,  as $\eps,\,a_1,\,a_2\rightarrow 0+$,  $\kappa_{\eps,a_1,a_2}\rightarrow \kappa_0$.  Assume that the above max is achieved at $x_*$.  Then we have:
\begin{prop}\label{p4.12}
Assume that $\eps,\,a_1,\,a_2$ are chosen so that $\frac{c_*+\rho(\eps)}{1+a_1}<\frac{c_*-\rho(\eps)}{1-a_2}$.  Assume also that both $(\varphi_1)^{\eps}$ and $(\varphi_2)_{\eps}$ are punctually second order differentiable at $x_*$. Then one has:
\begin{enumerate}
\item $\lambda[\frac{\chi-\rho(\eps)\omega_0}{1-a_2}+dd^c\frac{(\varphi_2)_{\eps}}{1-a_2}](x_*)\in \Gamma$,  
\item $F(\frac{\chi+\rho(\eps)\omega_0}{1+a_1}+dd^c\frac{(\varphi_1)^{\eps}}{1+a_1})(x_*)\le F(\frac{\chi-\rho(\eps)\omega_0}{1-a_2}+dd^c\frac{(\varphi_2)_{\eps}}{1-a_2})(x_*)$.
\end{enumerate}
In the above,  $c_*>0$ is the constant that allows one to write $\chi=\tilde{\chi}+c_*\omega_0$ with $\lambda(\tilde{\chi})\in \Gamma$.
\end{prop}
This proposition allows us to exclude the non-uniqueness of viscosity solutions,  as long as $(\varphi_1)^{\eps}$,  $(\varphi_2)_{\eps}$ are both punctually second order differentiable at $x_*$.

\begin{proof}
(of Theorem \ref{t4.1},  assuming $(\varphi_1)^{\eps}$ and $(\varphi_2)_{\eps}$ are punctually second order differentiable at $x_*$)
We choose $\eps,\,a_1,\,a_2$ so that $\frac{c_*+\rho(\eps)}{1+a_1}<\frac{c_*-\rho(\eps)}{1-a_2}$.  
Combining Proposition \ref{p4.3} and Proposition \ref{p4.12},  we see that at $x_*$:
\begin{equation}
\begin{split}
&f\big(\lambda[\frac{\chi+\rho(\eps)\omega_0}{1+a_1}+dd^c\frac{(\varphi_1)^{\eps}}{1+a_1}]\big)(x_*)\le f(\lambda[\frac{\chi-\rho(\eps)\omega_0}{1-a_2}+dd^c\frac{(\varphi_2)_{\eps}}{1-a_2}])(x_*),\\
&f\big(\lambda[\frac{\chi+\rho(\eps)\omega_0}{1+a_1}+dd^c\frac{(\varphi_1)^{\eps}}{1+a_1}]\big)(x_*)\ge e^{G(x_*,\frac{(\varphi_1)^{\eps}(x_*)}{1+a_1})}-\rho_1(\eps,a_1),\\
&f\big(\lambda[\frac{\chi-\rho(\eps)\omega_0}{1-a_2}+dd^c\frac{(\varphi_2)_{\eps}}{1-a_2}])\big)(x_*)\le e^{G(x_*,\frac{(\varphi_2)_{\eps}(x_*)}{1-a_2})}+\rho_2(\eps,a_2).
\end{split}
\end{equation}
Combining the three inequalities,  we see that:
\begin{equation}\label{4.21}
e^{G(x_*,\frac{(\varphi_1)^{\eps}(x_*)}{1+a_1})}-\rho_1(\eps,a_1)\le e^{G(x_*,\frac{(\varphi_2)_{\eps}(x_*)}{1-a_2})}+\rho_2(\eps,a_2).
\end{equation}
On the other hand,  $\frac{(\varphi_1)^{\eps}(x_*)}{1+a_1}-\frac{(\varphi_2)_{\eps}(x_*)}{1-a_2}=\kappa_{\eps,a_1,a_2}\rightarrow \kappa_0>0$ as $\eps,\,a_1,\,a_2\rightarrow 0$.  This is clearly inconsistent with (\ref{4.21}) when $\eps,\,a_1,\,a_2$ are all small enough.
\end{proof}
Now we prove Proposition \ref{p4.12}.
\begin{proof}
(Of Proposition \ref{p4.12}) Since both $(\varphi_1)^{\eps}$ and $(\varphi_2)_{\eps}$ are differentiable at $x_*$,  and that $\frac{(\varphi_1)^{\eps}}{1+a_1}-\frac{(\varphi_2)_{\eps}}{1-a_2}$ achieves maximum at $x_*$,  we see that:
\begin{equation*}
dd^c\frac{(\varphi_2)_{\eps}}{1-a_2}-dd^c\frac{(\varphi_1)^{\eps}}{1+a_1}\ge 0.
\end{equation*}
Moreover,
\begin{equation*}
\frac{\chi-\rho(\eps)\omega_0}{1-a_2}-\frac{\chi+\rho(\eps)\omega_0}{1+a_1}=\tilde{\chi}(\frac{1}{1-a_2}-\frac{1}{1+a_1})+(\frac{c_*-\rho(\eps)}{1-a_2}-\frac{c_*+\rho(\eps)}{1+a_1})\omega_0.
\end{equation*}
Therefore, 
\begin{equation}\label{4.22}
\frac{\chi-\rho(\eps)\omega_0}{1-a_2}+dd^c\frac{(\varphi_2)_{\eps}}{1-a_2}\ge \frac{\chi+\rho(\eps)\omega_0}{1+a_1}+dd^c\frac{(\varphi_1)^{\eps}}{1+a_1}+c_1\tilde{\chi}+\eta.
\end{equation}
In the above,  $c_1=\frac{1}{1-a_2}-\frac{1}{1+a_1}>0$,  $\lambda(\tilde{\chi})\in \Gamma$,  $\eta\ge 0$.  Also by Proposition \ref{p4.3},  part (1),  we also have $\lambda[\frac{\chi+\rho(\eps)\omega_0}{1+a_1}+dd^c\frac{(\varphi_1)^{\eps}}{1+a_1}]\in \Gamma$.  Therefore,  we see that $\lambda[\frac{\chi-\rho(\eps)\omega_0}{1-a_2}+dd^c\frac{(\varphi_2)_{\eps}}{1-a_2}]\in \Gamma$.

To prove the second part,  we wish to use concavity.  Indeed,  one has:
\begin{equation*}
\begin{split}
&f\big(\lambda[\frac{\chi+\rho(\eps)\omega_0}{1+a_1}+dd^c\frac{(\varphi_1)^{\eps}}{1+a_1}]\big)(x_*)\le f\big(\lambda[\frac{\chi-\rho(\eps)\omega_0}{1-a_2}+dd^c\frac{(\varphi_2)_{\eps}}{1-a_2}]\big)(x_*)\\
&+ \frac{\partial F}{\partial h_{i\bar{j}}}(\frac{\chi-\rho(\eps)}{1-a_2}+dd^c\frac{(\varphi_2)_{\eps}}{1-a_2})\big(\frac{\chi_{i\bar{j}}+\rho(\eps)g_{i\bar{j}}}{1+a_1}+\frac{(\varphi_1)_{i\bar{j}}^{\eps}}{1+a_1}\\
&-\frac{\chi_{i\bar{j}}-\rho(\eps)g_{i\bar{j}}}{1-a_2}-dd^c\frac{((\varphi_2)_{\eps})_{i\bar{j}}}{1-a_2}\big).
\end{split}
\end{equation*}
So we just need to show that:
\begin{equation*}
\begin{split}
&\frac{\partial F}{\partial h_{i\bar{j}}}(\frac{\chi-\rho(\eps)}{1-a_2}+dd^c\frac{(\varphi_2)_{\eps}}{1-a_2})\big(\frac{\chi_{i\bar{j}}+\rho(\eps)g_{i\bar{j}}}{1+a_1}+\frac{(\varphi_1)_{i\bar{j}}^{\eps}}{1+a_1}\\
&-\frac{\chi_{i\bar{j}}-\rho(\eps)g_{i\bar{j}}}{1-a_2}-\frac{((\varphi_2)_{\eps})_{i\bar{j}}}{1-a_2}\big)\le 0.
\end{split}
\end{equation*}
By Lemma \ref{l3.6N},  we just need to \sloppy show that $\lambda\big(\frac{\chi-\rho(\eps)\omega_0}{1-a_2}+dd^c\frac{(\varphi_2)_{\eps}}{1-a_2}-\frac{\chi+\rho(\eps)\omega_0}{1+a_1}-dd^c\frac{(\varphi_1)^{\eps}}{1+a_1}\big)\in \Gamma$.  However,  one can already see this from (\ref{4.22}).
\end{proof}

Next we look at the general case,  without assuming $(\varphi_1)^{\eps}$,  $(\varphi_2)_{\eps}$ punctually second order differentiable at $x_*$.  For this,  we need a perturbation argument from \cite{EGZ} and \cite{L}.

First we choose normal coordinate near $x_*$,  such that $x_*$ is given by $z=0$.  We wish to show that:
\begin{lem}\label{l4.13}
There exists a neighborhood $U_0$ of $x_*$,  and 
there exists a sequence $p_k\in \bC^n$,  $p_k\rightarrow 0$,  and a sequence $\delta_k>0,\,\delta_k\rightarrow 0$,  such that one can find a sequence of points $x_k\in U_0$ with the following properties hold:
\begin{enumerate}
\item $x_k\rightarrow x_*$ as $k\rightarrow \infty$,
\item $\frac{(\varphi_1)^{\eps}}{1+a_1}-\frac{(\varphi_2)_{\eps}}{1-a_2}-<p_k,z>-\delta_k|z|^2$ has local maximum at $x_k$,
\item Both $(\varphi_1)^{\eps}$ and $(\varphi_2)_{\eps}$ are punctually second order differentiable at $x_k$.
\end{enumerate}
\end{lem}
Another thing we note is that:
\begin{lem}\label{l4.14}
Define $\psi_{\eps,a_1,k}(z)=\frac{(\varphi_1)^{\eps}}{1+a_1}-<p_k,z>-\delta_k|z|^2$ on $U_0$.  Then for large enough $k$ ($\eps,\,a_1$ is fixed now) $\psi_{\eps,a_1,k}$ solves the following inequalities on $U_0$ in the viscosity sense:
\begin{equation*}
\begin{split}
&\lambda[\frac{\chi+2\rho(\eps)\omega_0}{1+a_1}+dd^c\psi_{\eps,a_1,k}]\in \Gamma.\\
&
f\big(\lambda[\frac{\chi+2\rho(\eps)\omega_0}{1+a_1}+dd^c\psi_{\eps,a_1,k}]\big)\ge e^{G(x,\psi_{\eps,a_1,k})}-2\rho_1(a_1,\eps).
\end{split}
\end{equation*}
\end{lem}

We first explain how this implies the uniqueness of viscosity solutions in the general case.
\begin{proof}
(Of Theorem \ref{t4.1},  in the general case)
We just need to suitably choose $\eps,\,a_1,\,a_2$,  and choose $k$ large enough,  then we evaluate at $x_k$.  Now we know that $\psi_{\eps,a_1,k}-\frac{(\varphi_2)_{\eps}}{1-a_2}$ achieves maximum at $x_k$.  Moreover,  both $\psi_{\eps,a_1,k}$ and $(\varphi_2)_{\eps}$ are punctually second order differentiable at $x_k$.  Therefore,  if we follow the argument of Proposition \ref{p4.12},  we see that if we choose $\eps,\,a_1,\,a_2$ so that $\frac{c_*+2\rho(\eps)}{1+a_1}<\frac{c_*-\rho(\eps)}{1-a_2}$,  we would be able to conclude that:
\begin{equation*}
f\big(\lambda[\frac{\chi+2\rho(\eps)\omega_0}{1+a_1}+dd^c\psi_{\eps,a_1,k}]\big)(x_k)\le f\big(\lambda[\frac{\chi-\rho(\eps)\omega_0}{1-a_2}+dd^c\frac{(\varphi_2)_{\eps}}{1-a_2}]\big)(x_k).
\end{equation*}
Therefore,  we see that:
\begin{equation}\label{4.23}
e^{G(x_k,\psi_{\eps,a_1,k}(x_k))}-2\rho_1(a_1,\eps)\le e^{G(x_k,\frac{(\varphi_2)_{\eps}}{1-a_2}(x_k))}+\rho_2(a_2,\eps).
\end{equation}
Now one passes to limit as $k\rightarrow \infty$.  Note that $\psi_{\eps,a_1,k}(x_k)\rightarrow \frac{(\varphi_1)^{\eps}}{1+a_1}(x_*)$,  we see that:
\begin{equation*}
e^{G(x_*,\frac{(\varphi_1)^{\eps}}{1+a_1}(x_*))}-2\rho_1(a_1,\eps)\le e^{G(x_*,\frac{(\varphi_2)_{\eps}}{1-a_2}(x_*))}+\rho_2(a_2,\eps).
\end{equation*}
On the other hand,  $\frac{(\varphi_1)^{\eps}}{1+a_1}(x_*)-\frac{(\varphi_2)_{\eps}}{1-a_2}(x_*)$ is strictly positive and bounded away from zero as $\eps,\,a_1,\,a_2\rightarrow 0$.  From the strict monotonicity of $G$,  we see a contradiction.
\end{proof}
Now it only remains to establish the technicalities Lemma \ref{l4.13} and \ref{l4.14}.  We start with Lemma \ref{l4.14}.

\begin{proof}
(Of Lemma \ref{l4.14})
Let $x_0\in U_0$ and let $P$ be a $C^2$ function on $U_0$ that touches $\psi_{\eps,a_1,k}$ from above at $x_0$.  This would imply that $P+<p_k,z>+\delta_k|z|^2$ touches $\frac{(\varphi_1)^{\eps}}{1+a_1}$ at $x_0$.  Therefore,  we may use Proposition \ref{p4.3},  part (1) to conclude that:
\begin{equation*}
\begin{split}
&\lambda[\frac{\chi+\rho(\eps)\omega_0}{1+a_1}+dd^c(P+<p_k,z>+\delta_k|z|^2)]\in \Gamma,\\
&f\big(\lambda[\frac{\chi+\rho(\eps)\omega_0}{1+a_1}+dd^c(P+<p_k,z>+\delta_k|z|^2]\big)\\
&\ge e^{G(x_0,(P+<p_k,z>+\delta_k|z|^2)(x_0))}-\rho_1(\eps,a_1).
\end{split}
\end{equation*}
On the other hand,  it is easy to see that when $k$ is large enough (so that $\delta_k$ is small enough),  one has:
\begin{equation*}
\begin{split}
&\frac{\chi+\rho(\eps)\omega_0}{1+a_1}+dd^c(P+<p_k,z>+\delta_k|z|^2)\le \frac{\chi+2\rho(\eps)\omega_0}{1+a_1}+dd^c P,\\
&e^{G(x_0,(P+<p_k,z>+\delta_k|z|^2)(x_0))}-\rho_1(\eps,a_1)\ge e^{G(x_0,P(x_0))}-2\rho_1(\eps,a_1).
\end{split}
\end{equation*}
Combining,  we get:
\begin{equation*}
\begin{split}
&\lambda[\frac{\chi+2\rho(\eps)\omega_0}{1+a_1}+dd^cP]\in \Gamma,\\
&f\big(\lambda[\frac{\chi+2\rho(\eps)\omega_0}{1+a_1}+dd^cP]\big)\ge e^{G(x_0,P(x_0))}-2\rho_1(\eps,a_1).
\end{split}
\end{equation*}
\end{proof}
Now we prove Lemma \ref{l4.13}.
\begin{proof}
(Of Lemma \ref{l4.13})  Without loss of generality,  let us assume that $U_0=B_1(0)$,  $x_*=0$ under the local coordinates.  Then this lemma really follows from a lemma of Crandall,  Ishii and Lions (\cite{CIL}  lemma A.3),  which states that:\\

Let $\varphi:\bR^N\rightarrow \bR$ be semi-convex and $\hat{x}$ be a strict local maximum of $\varphi$.  For $p\in \bR^N$,  put $\varphi_p(x)=\varphi(x)+<p,x>$.  Then for any $r>0,\,\delta>0$,  the following set $K$ has positive measure:
\begin{equation*}
K=\{x\in B_r(\hat{x}):\text{there exists $p\in B_{\delta}(0)$ such that $\varphi_p(x)$ has a local maximum at $x$}\}.
\end{equation*}

We are going to apply this lemma with $\varphi=\frac{(\varphi_1)^{\eps}}{1+a_1}-\frac{(\varphi_2)_{\eps}}{1-a_2}-\frac{1}{k}|z|^2$.  This function will be semi-convex due to Lemma \ref{l4.11N}.  Moreover,   we know that $\frac{(\varphi_1)^{\eps}}{1+a_1}-\frac{(\varphi_2)_{\eps}}{1-a_2}-\frac{1}{k}|z|^2$ has strict maximum at $x_*$ (given by $z=0$).  Then the above lemma applies.

We will also choose $\delta=\frac{1}{k},\,r=\frac{1}{k}$.  Then we can conclude the set of $x\in B_{\frac{1}{k}}(x_*)$ such that $\frac{(\varphi_1)^{\eps}}{1+a_1}-\frac{(\varphi_2)_{\eps}}{1-a_2}-\frac{1}{k}|z|^2+<p,z>$ has local maximum at $x$ for some $|p|\le \frac{1}{k}$ has positive measure.  Since $(\varphi_1)^{\eps}$ and $(\varphi_2)_{\eps}$ are both punctually second order differentiable a.e.,  we can find $x_k\in B_{\frac{1}{k}}(x_*)$ belonging to the above said set,  such that both $(\varphi_1)^{\eps}$ and $(\varphi_2)_{\eps}$ are punctually second order differentiable at $x_k$.  If we denote the corresponding $p$ to be $p_k$,  we see that we are done.
\end{proof}

\subsection{when the right hand side does not depend on $\varphi$}
In this subsection,  we will assume that the right hand side $G$ depends only on $x$.  In this case,  one can only hope to solve $F(\chi+dd^c\varphi)=e^{G+c}$,  for some constant $c$.
Even though we can show that there is a unique $c\in \bR$ that makes this equation solvable in the viscosity sense,  we still don't have a good enough understanding of this constant,  which is the main hurdle to a proof of uniqueness of viscosity solutions in this case.
Let us start with observing the monotonicity of the constant,  in terms of the right hand side.  More specifically,  we have:
\begin{prop}\label{p4.15}
Let $G_1,\,G_2\in C(M)$ with $G_1\ge G_2$.  Assume that there exist $c_i\in \bR,\,\varphi_i\in C(M),\,i=1,\,2$,  that solve $F(\chi+dd^c\varphi_i)=e^{G_i+c_i}$ in the viscosity sense.  Then one has $c_1\le c_2$.
\end{prop}
\begin{proof}
Assume that this is false,  namely $c_1>c_2$.  This would imply that $e^{G_1+c_1}>e^{G_2+c_2}+\delta_0$ on $M$,  for some $\delta_0>0$.  Heuristically this would lead to a contradiction if one evaluates at $x_0$,  where $\varphi_1-\varphi_2$ achieves maximum.  At this point,  one would have $\chi+dd^c\varphi_1\le \chi+dd^c\varphi_2$,  so that $F(\chi+dd^c\varphi_1)\le F(\chi+dd^c\varphi_2)$ at $x_0$.  

To proceed rigorously,  one needs to perform the super/inf convolutions considered in Subsection 4.1.  One consideres $\frac{(\varphi_1)^{\eps}}{1+a_1}-\frac{(\varphi_2)_{\eps}}{1-a_2}$,  where $(\varphi_1)^{\eps}$ and $(\varphi_2)_{\eps}$ are defined according to (\ref{4.1}) and (\ref{4.2}).  

Assume that the maximum is achieved at $x_*$.  If both $(\varphi_1)^{\eps}$ and $(\varphi_2)_{\eps}$ are both twice differentiable at $x_*$,  then one has:
\begin{equation*}
\begin{split}
&f\big(\lambda[\frac{\chi+\rho(\eps)\omega_0}{1+a_1}+dd^c\frac{(\varphi_1)^{\eps}}{1+a_1}]\big)(x_*)\ge e^{G_1(x_*)+c_1}-\rho_1(\eps,a_1),\\
&f\big(\lambda[\frac{\chi-\rho(\eps)\omega_0}{1-a_2}+dd^c\frac{(\varphi_2)_{\eps}}{1-a_2}]\big)(x_*)\le e^{G_2(x_*)+c_2}+\rho_2(\eps,a_2).
\end{split}
\end{equation*}
Since $\frac{(\varphi_1)^{\eps}}{1+a_1}-\frac{(\varphi_2)_{\eps}}{1-a_2}$ has maximum at $x_*$,  we see that,  if $\eps,\,a_1,\,a_2$ are chosen so that $\frac{c_*+\rho(\eps)}{1+a_1}<\frac{c_*-\rho(\eps)}{1-a_2}$,  we can follow the argument of Proposition \ref{p4.12} to see that:
\begin{equation*}
f\big(\lambda[\frac{\chi+\rho(\eps)\omega_0}{1+a_1}+dd^c\frac{(\varphi_1)^{\eps}}{1+a_1}]\big)(x_*)\le f\big(\lambda[\frac{\chi-\rho(\eps)\omega_0}{1-a_2}+dd^c\frac{(\varphi_2)_{\eps}}{1-a_2}]\big)(x_*).
\end{equation*}
Then one gets: $e^{G_1(x_*)+c_1}-\rho_1(\eps,a_1)\le e^{G_2(x_*)+c_2}+\rho_2(\eps,a_2)$.  This is inconsistent with $e^{G_1+c_1}\ge e^{G_2+c_2}+\delta_0$,  if one chooses $\eps,\,a_1,\,a_2$ all small enough,  and $\frac{c_*+\rho(\eps)}{1+a_1}<\frac{c_*-\rho(\eps)}{1-a_2}$.

In the general case,  we can take a coordinate chart in a neighborhood of $x_*$,  and consider $\frac{(\varphi_1)^{\eps}}{1+a_1}-\frac{(\varphi_2)_{\eps}}{1-a_2}-<p_k,z>-\delta_k|z|^2$ with $p_k,\,\delta_k\rightarrow 0$.  Using the same argument as in the proof of Theorem \ref{t4.1} in the general case,  we can find a sequence $x_k\rightarrow x_*$ such that both $(\varphi_1)^{\eps}$ and $(\varphi_2)_{\eps}$ are punctually second order differentiable at $x_k$,  and the above function has local minimum at $x_k$.  We still get a contradiction after evaluating at $x_k$ and passing to the limit as $k\rightarrow \infty$.
\end{proof}
A direct consequence of the above proposition is that there is a unique constant $c$ that allows for a viscosity solutions:
\begin{cor}
Let $G\in C(M)$,  there is at most one constant $c\in \bR$,  such that $F(\chi+dd^c\varphi)=e^{G+c}$ is solvable in the viscosity sense.
\end{cor}
Because of this result,  we may simply denote this constant to be $c(G)$.  Proposition \ref{p4.15} then implies that $c(G_1)\le c(G_2)$ whenever $G_1\ge G_2$.

The question that is of crucial importance to us is the following:
\begin{q}\label{q4.17}
Assume that $G_1\ge G_2$,  and $G_1\neq G_2$.  Do we actually have $c(G_1)<c(G_2)$?
\end{q}
In order to justify its importance,  we are going to show the uniqueness of viscosity solutions,  assuming we have an affirmative answer to Question \ref{q4.17}.  More precisely:
\begin{thm}\label{t4.2}
Let $G\in C(M)$.  Assume that for any $G'\in C(M),\,G'\le G$ and $G'\neq G$,  one has $c(G')>c(G)$.  Then there is at most one viscosity solutions to $F(\chi+dd^c\varphi)=e^{G+c}$.
\end{thm}

Before presenting the proof in full rigor,  let us first explain heuristically how the above said strict monotonicity helps.

Let $\varphi_1$ and $\varphi_2$ be two viscosity solutions to the equation.  Denote $E=\{x\in M:\varphi_2(x)-\varphi_1(x)=\min_M(\varphi_2-\varphi_1)\}$.  Clearly $E$ is a compact subset of $M$ and we will be done if we can show $E=M$.  Assume otherwise,  we can take $\delta>0$ small enough,  such that $E_{\delta}\neq M$,  where $E_{\delta}$ denotes the $\delta$-neighborhood of $E$.

Now we can define $\tilde{G}\in C(M)$ as follows: $e^{\tilde{G}}=e^G$ on $E_{\frac{\delta}{2}}$,  $e^{\tilde{G}}=\frac{1}{2}e^G$ outside $E_{\delta}$,  and $\tilde{G}\le G$ on $M$.  Let $\eta$ be a viscosity solution to:
\begin{equation*}
F(\chi+dd^c\eta)=e^{\tilde{G}+\tilde{c}},\,\,\sup_M\eta=0.
\end{equation*}
From the previous section,  we know such $\eta$ exists.  
From the assumption,  we know that $\tilde{c}>c$.  Let $0<r<1$,  we consider the minimum of $\varphi_2-((1-r)\varphi_1+r\eta)$ and assume that it is achieved at $x_{\delta,r}$.  Evaluating at $x_{\delta,r}$,  we would have:
\begin{equation*}
f\big(\lambda[\chi+dd^c\varphi_2]\big)\ge f\big(\lambda[\chi+dd^c(1-r)\varphi_1+r\eta]\big)\ge (1-r)f(\lambda[\chi+dd^c\varphi_1])+rf(\lambda[\chi+dd^c\eta]).
\end{equation*}
This would mean that:
\begin{equation}\label{4.24}
e^{G+c}(x_{\delta,r})\ge (1-r)e^{G+c}(x_{\delta,r})+re^{\tilde{G}+\tilde{c}}(x_{\delta,r}).
\end{equation}
On the other hand,  for fixed $\delta$,  the minimum of $\varphi_2-((1-r)\varphi_1+r\eta)$ can only be achieved in $E_{\frac{\delta}{2}}$,  as long as $r$ is small enough,  but then one would have $e^{\tilde{G}}(x_{\delta,r})=e^{G}(x_{\delta,r})$,  and this is inconsistent with (\ref{4.24}).

Next we are going to make the above argument rigorous.  

As before,  in place of $\varphi_1$,  we wish to consider $\frac{(\varphi_1)^{\eps}}{1+a_1}$,  $\frac{(\varphi_2)_{\eps}}{1-a_2},\,\frac{\eta^{\eps}}{1-a_2}$,  with parameters $\eps,\,a_i,\,i=1,\,2$ small enough,  then similar arguments as before shows the following:
\begin{lem}\label{l4.18}
There exist continuous functions $\rho(\eps),\,\rho_i(\eps,a_i),\,i=1,2$ with $\rho(0+)=0,\,\rho_i(0+,0+)=0,\,i=1,2$,  such that:
\begin{enumerate}
\item $\frac{(\varphi_1)^{\eps}}{1+a_1}$ solves the following in the viscosity sense:
\begin{equation*}
\begin{split}
&\lambda[\frac{\chi+\rho(\eps)\omega_0}{1+a_1}+dd^c\frac{(\varphi_1)^{\eps}}{1+a_1}]\in \Gamma,\\
&F(\frac{\chi+\rho(\eps)\omega_0}{1+a_1}+dd^c\frac{(\varphi_1)^{\eps}}{1+a_1})\ge e^{G+c}-\rho_1(\eps,a_1).
\end{split}
\end{equation*}
\item $\frac{(\varphi_2)_{\eps}}{1-a_2}$ solves $F(\frac{\chi-\rho(\eps)}{1-a_2}+dd^c\frac{(\varphi_2)_{\eps}}{1-a_2})\le e^{G+c}+\rho_2(\eps,a_2)$ in the viscosity sense.  
\item $\frac{\eta^{\eps}}{1+a_1}$ solves $\lambda[\frac{\chi+\rho(\eps)\omega_0}{1+a_1}+dd^c\frac{\eta^{\eps}}{1+a_1}]\in \Gamma$ and $F(\frac{\chi+\rho(\eps)\omega_0}{1+a_1}+dd^c\frac{\eta^{\eps}}{1+a_1})\ge e^{\tilde{G}+\tilde{c}}-\rho_2(\eps,a_1)$ in the viscosity sense.
\end{enumerate}
\end{lem}
Next we consider the maximum point of $((1-r)\frac{(\varphi_1)^{\eps}}{1+a_1}+r\frac{\eta^{\eps}}{1+a_1})-\frac{(\varphi_2)_{\eps}}{1-a_2}$.  Assume that the maximum is achieved at $x_*$,  then we have:
\begin{lem}
For any fixed $\delta>0$,  $x_*\in E_{\frac{\delta}{2}}$ as long as $\eps,\,a_1,\,a_2,\,r$ are all small enough.
\end{lem}
\begin{proof}
We observe that for $x\notin E_{\frac{\delta}{2}}$,  there exists $\delta'>0$,  such that 
\begin{equation*}
\varphi_1-\varphi_2\le \max_M(\varphi_1-\varphi_2)-\delta',
\end{equation*}
On the other hand,  we know that as $\eps,\,a_1,\,a_2,\,r\rightarrow 0$,  we have that $((1-r)\frac{(\varphi_1)^{\eps}}{1+a_1}+r\frac{\eta^{\eps}}{1+a_1})-\frac{(\varphi_2)_{\eps}}{1-a_2}$ will converge to $\varphi_1-\varphi_2$ uniformly,  so if the parameters are all small enough and $x\notin E_{\frac{\delta}{2}}$,  one would have:
\begin{equation*}
\big((1-r)\frac{(\varphi_1)^{\eps}}{1+a_1}+r\frac{\eta^{\eps}}{1+a_1}\big)-\frac{(\varphi_2)_{\eps}}{1-a_2}\le \max_M\big(\big((1-r)\frac{(\varphi_1)^{\eps}}{1+a_1}+r\frac{\eta^{\eps}}{1+a_1}\big)-\frac{(\varphi_2)_{\eps}}{1-a_2}\big)-\frac{\delta'}{2}.
\end{equation*}
This proves $x_*\in E_{\frac{\delta}{2}}$.
\end{proof}
In order to fix the issue that the functions $(\varphi_1)^{\eps},\,\eta^{\eps},\,(\varphi_2)_{\eps}$ are not punctually second order differentiable at $x_*$,  we need to consider slight perturbations of the original function.  As before,  we take a coordinate chart near $x_*$,  and we have the following analogue of Lemma \ref{l4.13}:
\begin{lem}
There exists a neighborhood $U_0$ of $x_*$,  and there exists a sequence $p_k\in \bC^n$,  $p_k\rightarrow 0$,  and a sequence $\delta_k>0,\,\delta_k\rightarrow 0$,  such that one can find a sequence of $x_k\in U_0$,  such that if we define $\psi_{\eps,a_1,k}:=\frac{(\varphi_1)^{\eps}}{1+a_1}-<p_k,z>-\delta_k|z|^2$,  we have:
\begin{enumerate}
\item $x_k\rightarrow x_*$ as $k\rightarrow \infty$,
\item $((1-r)\psi_{\eps,a_1,k}+r\frac{\eta^{\eps}}{1+a_1})-\frac{(\varphi_2)_{\eps}}{1-a_2}$ has local maximum at $x_k$,
\item $(\varphi_1)^{\eps},\,\eta^{\eps},\,(\varphi_2)_{\eps}$ are all punctually second order differentiable at $x_k$.
\end{enumerate}
\end{lem}

With this preparation,  we are ready to present the rigorous proof of Theorem \ref{t4.2}:
\begin{proof}
(Of Theorem \ref{t4.2}) 
From Lemma \ref{l4.18} and the definition of $\psi_{\eps,a_1,k}$,  we see that $\psi_{\eps,a_1,k}$ solves the following equation in the viscosity sense:
\begin{equation}\label{4.25}
\begin{split}
&\lambda[\frac{\chi+2\rho(\eps)\omega_0}{1+a_1}+dd^c\psi_{\eps,a_1,k}]\in\Gamma,\\
&f\big(\lambda[\frac{\chi+2\rho(\eps)\omega_0}{1+a_1}+dd^c\psi_{\eps,a_1,k}]\big)\ge e^{G+c}-\rho_1(\eps,a_1).
\end{split}
\end{equation}
From Lemma \ref{l4.18},  part (3),  we also have:
\begin{equation}\label{4.26}
\begin{split}
&\lambda[\frac{\chi+2\rho(\eps)\omega_0}{1+a_1}+dd^c\frac{\eta^{\eps}}{1+a_1}]\in \Gamma,\\
&f\big(\lambda[\frac{\chi+2\rho(\eps)\omega_0}{1+a_1}+dd^c\frac{\eta^{\eps}}{1+a_1}]\big)\ge e^{\tilde{G}+\tilde{c}}-\rho_2(\eps,a_1).
\end{split}
\end{equation}
Since $x_k\rightarrow x_*$ and $x_*\in E_{\frac{\delta}{2}}$,  we see that $x_k\in E_{\frac{\delta}{2}}$ for large enough $k$.  Moreover,  since $\psi_{\eps,a_1,k}$ and $\eta^{\eps}$ are punctually second order differentiable at $x_k$,  we can evaluate (\ref{4.25}),  (\ref{4.26}) at $x_k$ and see that:
\begin{equation}
\begin{split}
&f\big(\lambda[\frac{\chi+2\rho(\eps)\omega_0}{1+a_1}+dd^c((1-r)\psi_{\eps,a_1,k}+r\eta^{\eps})]\big)(x_k)\\
&\ge (1-r)f\big(\lambda[\frac{\chi+2\rho(\eps)\omega_0}{1+a_1}+dd^c\psi_{\eps,a_1,k})])(x_k)+rf\big(\lambda[\frac{\chi+2\rho(\eps)\omega_0}{1+a_1}+dd^c\eta^{\eps}]\big)(x_k)-2\rho_1(a_1,\eps)\\
&\ge (1-r)e^{G+c}(x_k)+re^{\tilde{G}+\tilde{c}}(x_k)-2\rho_1(a_1,\eps)=(1-r)e^{G+c}(x_k)+re^{G+\tilde{c}}(x_k)-2\rho_1(a_1,\eps).
\end{split}
\end{equation}
In the first inequality above,  we used the concavity of $f$.  In the second inequality above,  we used that $\psi_{\eps,a_1,k}$ and $\eta^{\eps}$ are subsolutions.  Also since $x_k\in E_{\frac{\delta}{2}}$,  we have $\tilde{G}(x_k)=G(x_k)$.
On the other hand,  if one chooses the parameters such that:
\begin{equation*}
\frac{c_*+2\rho(\eps)}{1+a_1}\le \frac{c_*-\rho(\eps)}{1-a_2},
\end{equation*}
we then have (similar to Proposition \ref{p4.12}):
\begin{equation*}
\begin{split}
&f\big(\lambda[\frac{\chi+2\rho(\eps)\omega_0}{1+a_1}+dd^c((1-r)\psi_{\eps,a_1,k}+r\eta^{\eps})]\big)(x_k)\le f\big(\lambda[\frac{\chi-\rho(\eps)\omega_0}{1-a_2}+dd^c\frac{(\varphi_2)_{\eps}}{1-a_2}]\big)(x_k)\\
&\le e^{G+c}(x_k)+\rho_2(\eps,a_2).
\end{split}
\end{equation*}
Therefore,  we finally obtain that:
\begin{equation}\label{4.29}
(1-r)e^{G+c}(x_k)+re^{G+\tilde{c}}(x_k)-2\rho_1(a_1,\eps)\le e^{G+c}(x_k)+\rho_2(\eps,a_2).
\end{equation}
First we let $k\rightarrow \infty$ and use that $x_k\rightarrow x_*$,  we see that
\begin{equation*}
(1-r)e^{G+c}(x_*)+re^{G+\tilde{c}}(x_*)-2\rho_1(a_1,\eps)\le e^{G+c}(x_*)+\rho_2(\eps,a_2).
\end{equation*}
Note that if $r>0$ and fixed,  $e^{G+c}(x_*)$ is strictly less than $(1-r)e^{G+c}(x_*)+re^{G+\tilde{c}}(x_*)$.  Thereby we get a contradiction with (\ref{4.29}) after $\eps,\,a_1,\,a_2$ are chosen sufficiently small.
\end{proof}
Now let us go back to Question \ref{q4.17}.  The only positive examples we know are the classical ones: if $\omega_0$ is a K\"ahler metric and if $\chi$ is closed,  so the equation could be written as:
\begin{equation*}
\big(\frac{(\chi+dd^c\varphi)^k\wedge \omega_0^{n-k}}{\omega_0^n}\big)^{\frac{1}{k}}=e^{G+c},\,\,\,1\le k\le n.
\end{equation*}
Then we can see that:
\begin{equation*}
e^{kc}=\frac{\int_M\chi^k\wedge \omega_0^{n-k}}{\int_Me^{kG}\omega_0^n}.
\end{equation*}
Therefore,  we see that Question \ref{q4.17} indeed holds in this case.  Another observation we make is that,  if either one of them is smooth,  then Question \ref{q4.17} also has an affirmative answer:
\begin{lem}
If either $G_1$ or $G_2$ is smooth,  then the answer to Question \ref{q4.17} is yes.
\end{lem}
\begin{proof}
First we assume that both $G_1$ and $G_2$ are smooth.  Assume otherwise,  that is,  $G_1$ and $G_2$ are both smooth,  $G_1\ge G_2$,  $G_1\neq G_2$,  but still $c(G_1)=c(G_2)$.  Let us denote this constant to be $c$.  Then we have:
\begin{equation*}
F(\chi+dd^c\varphi_1)=e^{G_1+c},\,\,F(\chi+dd^c\varphi_2)=e^{G_2+c}.
\end{equation*}
Subtracting and using the concavity of $F$,  we get:
\begin{equation*}
e^{G_2+c}-e^{G_1+c}=F(\chi+dd^c\varphi_2)-F(\chi+dd^c\varphi_1)\ge \frac{\partial F}{\partial h_{i\bar{j}}}(\chi+dd^c\varphi_2)(\varphi_2-\varphi_1)_{i\bar{j}}.
\end{equation*}

Therefore,  if we define 
\begin{equation*}
\tilde{\Omega}=\big(\det g\det(\frac{\partial F}{\partial h_{i\bar{j}}})(\chi+dd^c\varphi_2)\big)^{\frac{1}{n-1}}\Omega,
\end{equation*}
where $\Omega$ is defined in (\ref{2.1NN}),  
then the above equation can be written as:
\begin{equation*}
\frac{\tilde{\Omega}^{n-1}}{(n-1)!}\wedge dd^c(\varphi_2 - \varphi_1) \le (e^{G_2+c}- e^{G_1 +c})\frac{\omega_0^n}{n!}.
\end{equation*}
Now let $v$ be the Gauduchon factor corresponding to $\tilde{\Omega}$,  that is $dd^c\big(e^{(n-1)v}\tilde{\Omega}^{n-1}\big)=0$,  then we have:
\begin{equation*}
0=\int_Me^{(n-1)v}\frac{\tilde{\Omega}^{n-1}}{(n-1)!}\wedge dd^c(\varphi_2 - \varphi_1)=\int_Me^{(n-1)v +c}(e^{G_2}-e^{G_1})\frac{\omega_0^n}{n!}\le 0.
\end{equation*}
Here we use the fact that $c(G_1)=c(G_2)$ and $G_1 \ge G_2$. This implies $G_1=G_2$,  a contradiction.

Next we consider when only one of $G_1$ or $G_2$ is smooth.  Assume,  say,  $G_1$ is smooth,  $G_1\ge G_2$,  and $G_1>G_2$ on some open set $U$.  By adding a bump function supported on $U$,  it is easy to find a smooth function $G_2'$ such that $G_1\ge G_2'\ge G_2$,  and $G_1\neq G_2'$.  Then from the strict monotonicity in the smooth case,  we can conclude that:
\begin{equation*}
c(G_1)<c(G_2')\le c(G_2).
\end{equation*}
The other possibility that $G_2$ is smooth can be dealt with similarly.
\end{proof}

\section{Acknowlegement}
This work is supported in part by Simons Foundation Award with ID:605796.

\end{document}